\documentclass[reqno,12pt]{amsart}

\usepackage{graphicx}
\usepackage{amsmath}
\usepackage{mathabx}
\usepackage{amssymb}
\usepackage{dsfont}
\usepackage{enumitem}      

\usepackage[headings]{fullpage}

\numberwithin{equation}{section}
\def\RR{{\mathbb R}}
\def\MM{{\mathbb M}}

\def\LL{{\mathbb L}}

\def\GG{{\mathbb G}}

\def\eps{\varepsilon}

\def\cupp{\mathop{\cup}}

\def\ell{l}

\def\mint{{{\bf-}\!\!\!\!\!\hspace{-.1em}\int}}  
\def\mmint{{{\bf-}\!\!\!\!\hspace{-.1em}\int}} 
\def\limsup{\mathop{\overline{\lim}}}
\def\liminf{\mathop{\underline{\lim}}}

\def\DiM{\kappa}

\def\endproof{$\blacksquare$}

\newtheorem{theorem}{Theorem}[section]
\newtheorem{lemma}[theorem]{Lemma}
\newtheorem{proposition}[theorem]{Proposition}
\newtheorem{corollary}[theorem]{Corollary}

\theoremstyle{definition}
\newtheorem{definition}[theorem]{Definition}

\theoremstyle{remark}
\newtheorem{remark}[theorem]{Remark}

\title[Relaxation of nonconvex unbounded integrals in Cheeger-Sobolev spaces]{Relaxation of nonconvex unbounded integrals with general growth conditions in Cheeger-Sobolev spaces}

\author{\sc Omar Anza Hafsa}
\address{{(\rm Omar Anza Hafsa)} UNIVERSITE DE NIMES, Laboratoire MIPA, Site des Carmes, Place Gabriel P\'eri, 30021 N\^\i mes, France and LMGC, UMR-CNRS 5508, Place Eug\`ene Bataillon, 34095 Montpellier, France.}
\email{omar.anza-hafsa@unimes.fr}

\author{Jean-Philippe Mandallena}
\address{{(\rm Jean-Philippe Mandallena)} UNIVERSITE DE NIMES, Laboratoire MIPA, Site des Carmes, Place Gabriel P\'eri, 30021 N\^\i mes, France.}
\email{jean-philippe.mandallena@unimes.fr}


\begin{document}

\begin{abstract}
We study relaxation of nonconvex integrals of the calculus of variations in the setting of Cheeger-Sobolev spaces when the  integrand has not polynomial growth and can take infinite values.

\end{abstract}

\maketitle


\section{Introduction}

In this paper we are concerned with relaxation of integrals of type
\begin{equation}\label{InTRo-Eq1}
\int_X L(x,\nabla_\mu u(x))d\mu(x),
\end{equation}
where $(X,d,\mu)$ is a metric measure space, with $(X,d)$ a length space separable and compact, satisfying a weak $(1,p)$-Poincar\'e inequality with $p>1$ and such that $\mu$ is a doubling positive Radon measure on $X$, $u\in W^{1,p}_\mu(X;\RR^m)$ with $m\geq 1$ an integer and $\nabla_\mu u$ is the $\mu$-gradient of $u$, and $L:X\times\MM\to[0,\infty]$ is a Borel measurable integrand not necessarily convex with respect to $\xi\in\MM$, where $\MM$ denotes the space of all $m\times N$ matrices with $N\geq 1$ an integer. Such a relaxation problem in such a metric measure setting was studied for the first time in \cite{AHM15} (see also \cite{boubusep97,jpm00,oah-jpm03,fragala03,oah-jpm04,jpm05,mocanu05,HakKinPaLe14} and the references therein) when $L$ has $p$-growth, i.e., there exist $\alpha,\beta>0$ such that for every $x\in X$ and every $\xi\in\MM$,
\begin{equation}\label{InTRo-Eq2}
\alpha|\xi|^p\leq L(x,\xi)\leq\beta(1+|\xi|^p),
\end{equation}
where it is proved (see \cite[Theorem 2.21 and Corollary 2.27]{AHM15}) that if \eqref{InTRo-Eq2} holds then the relaxation of \eqref{InTRo-Eq1} with respect to the norm of $L^p_\mu(X;\RR^m)$ is given by  
$$
\int_X\mathcal{Q}_\mu L(x,\nabla_\mu u(x))d\mu(x),
$$
where $u\in W^{1,p}_\mu(X;\RR^m)$ and $\mathcal{Q}_\mu L:X\times \MM\to[0,\infty]$, called the $\mu$-quasiconvexification of $L$, is defined by
$$
\mathcal{Q}_\mu L(x,\xi):=\limsup_{\rho\to0}\inf\left\{\mint_{Q_\rho(x)}L(y,\xi+\nabla_\mu w(y))d\mu(y):w\in W^{1,p}_{\mu,0}(Q_\rho(x);\RR^m)\right\}.
$$
where $Q_\rho(x)$ is the open ball centered at $x\in X$ with radius $\rho>0$, and, for each open set $A$ of  $X$, $W^{1,p}_{\mu,0}(A;\RR^m)$ is the closure of ${\rm Lip}_0(A;\RR^m)$ with respect to $W^{1,p}_\mu$-norm with ${\rm Lip}_0(A;\RR^m):=\big\{u\in{\rm Lip}( X;\RR^m):u=0\hbox{ on } X\setminus A\big\}$, where ${\rm Lip}( X;\RR^m)$ is the class of Lipschitz functions from $ X$ to $\RR^m$ (see \S 3.1 for more details). 

Our motivation for developing relaxation, and more generally calculus of variations, in the setting of metric measure spaces comes from applications to hyperelasticity. In fact, the interest of considering a general measure is that its support can modeled a hyperelastic structure together with its singularities like for example thin dimensions, corners, junctions, etc. Such mechanical singular objects naturally lead to develop calculus of variations in the setting of metric measure spaces. In this way, having in mind the two basic conditions of hyperelasticity, i.e., ``the non-interpenetration of the matter" and ``the necessity of an infinite amount of energy to compress a finite piece of matter into a point",  it is then of interest to study relaxation of nonconvex integrals of type \eqref{InTRo-Eq1} when the integrand has not $p$-growth and can take infinite values: this is the general purpose of the present paper. 

For related works in the Euclidean case, i.e., when $(X,d,\mu)=(\overline{\Omega},|\cdot-\cdot|,\mathcal{L}_N)$ where $\Omega$ is a bounded subset of $\RR^N$ and $\mathcal{L}_N$ is the Lebesgue measure, we refer the reader to \cite{sychev04, sychev05, oah-jpm07, oah-jpm08a, sychev10, oah10, oah-jpm12, jpm13, conti-dolzmann15, jpm-ms16a, jpm-ms16b} and the references therein.  

\medskip

Generally speaking, in this paper our main contribution (see \S2.1, \S2.2 and \S2.3 for more details and more precisely Theorem \ref{MainTheorem}) is to prove that for $p>\DiM$, with $\DiM:={\ln(C_{d})\over \ln(2)}$ where $C_d\geq 1$ is the doubling constant, see \eqref{D-meas}, if $L$ is {\em radially uniformly upper semicontinuous}, i.e., there exists $a\in L^1_\mu(X;]0,\infty])$ such that
$$
\limsup_{t\to1^-}\sup_{x\in X}\sup_{\xi\in \LL_x}{L(x,t\xi)-L(x,\xi)\over a(x)+L(x,\xi)}\leq 0,
$$

where $\LL_x$ denotes the effective domain of $L(x,\cdot)$, and if $L$ has {\em$G$-growth}, i.e., there exist $\alpha,\beta>0$ such that for every $x\in X$ and every $\xi\in\MM$,
$$
\alpha G(x,\xi)\leq L(x,\xi)\leq \beta (1+G(x,\xi)),
$$
where $G:X\times\MM\to[0,\infty]$ is a Borel measurable and $p$-coercive integrand satisfying some ``convexity" assumptions, see \eqref{growth-on-Gx}, \eqref{hyp-convexity}, \eqref{lsc-intGx} and \eqref{growth-on-Gx-2}, then the relaxation of \eqref{InTRo-Eq1} with respect to the norm of $L^p_\mu(X;\RR^m)$ is given by  
$$
\int_X \widehat{\mathcal{Q}_\mu L}(x,\nabla_\mu u(x))d\mu(x),
$$
where $\widehat{\mathcal{Q}_\mu L}:X\times\MM\to[0,\infty]$ is defined by
$$
\widehat{\mathcal{Q}_\mu L}(x,\xi):=\liminf_{t\to 1^-}{\mathcal{Q}_\mu L}(x,t\xi).
$$
Moreover, we also prove (see Corollary \ref{CoRoLLaRY---MainTheorem}) that if, in addition, $\mathcal{Q}_\mu G=G$ and, for each $x\in X$, $\mathcal{Q}_\mu L(x,\cdot)$ is lower semicontinuous on the interior ${\rm int}(\LL_x)$ of $\LL_x$, then
$$
\widehat{\mathcal{Q}_\mu L}(x,\xi)=\left\{
\begin{array}{ll}
\mathcal{Q}_\mu L(x,\xi)&\hbox{if }x\in X\hbox{ and }\xi\in{\rm int}(\LL_x)\\
\lim\limits_{t\to 1^-}{\mathcal{Q}_\mu L}(x,t\xi)&\hbox{if }x\in X\hbox{ and }\xi\in\partial\LL_x\\
\infty&\hbox{otherwise.}
\end{array}
\right.
$$

\medskip

The plan of the paper is as follows. The main result (see Theorem \ref{MainTheorem} and also Corollary \ref{CoRoLLaRY---MainTheorem}) is given in Section 2. The proof of Theorem \ref{MainTheorem} is established in Section 4, whereas Corollary \ref{CoRoLLaRY---MainTheorem}, which is a consequence of Theorem \ref{MainTheorem}, is proved at the end of Section 2. Section 3 is devoted to several auxiliary results needed for proving Theorem \ref{MainTheorem}.

\subsubsection*{Notation} The open and closed balls centered at $x\in X$ with radius $\rho>0$ are denoted by:
\begin{trivlist}
\item[] $Q_\rho(x):=\Big\{y\in X:d(x,y)<\rho\Big\};$
\item[] $\overline{Q}_\rho(x):=\Big\{y\in X:d(x,y)\leq\rho\Big\}.$
\end{trivlist}
For $x\in X$ and $\rho>0$ we set 
$$
\partial Q_\rho(x):=\overline{Q}_\rho(x)\setminus Q_\rho(x)=\Big\{y\in X:d(x,y)=\rho\Big\}.
$$
For $A\subset X$, the diameter of $A$ (resp. the distance from a point $x\in X$ to the subset $A$) is defined by ${\rm diam}(A):=\sup_{x,y\in A}d(x,y)$ (resp. ${\rm dist}(x,A):=\inf_{y\in A}d(x,y)$).

The symbol $\mmint$ stands for the mean-value integral 
$$
\mint_Bf d\mu={1\over\mu(B)}\int_Bf d\mu.
$$


\section{Main results}

\subsection{Setting of the problem} Let $(X,d,\mu)$ be a metric measure space, where $(X,d)$ is a length space which is separable and compact, and $\mu$ is a positive Radon measure on $X$. In what follows, we assume that $\mu$ is doubling, i.e., there exists a constant $C_d\geq 1$ (called doubling constant) such that 
\begin{equation}\label{D-meas}
\mu\left(Q_\rho(x)\right)\leq C_d \mu\left(Q_{\rho\over 2}(x)\right)
\end{equation}
for all $x\in X$ and all $\rho>0$, and $X$ supports a weak $(1,p)$-Poincar\'e inequality with $1<p<\infty$, i.e., there exist $C_P>0$ and $\sigma\geq 1$ such that for every $x\in X$ and every $\rho>0$, 
\begin{equation}\label{1-p-PI}
\mint_{Q_\rho(x)}\left|f- \mint_{Q_\rho(x)} f d\mu\right|d\mu\leq \rho C_P\left(\mint_{Q_{\sigma\rho}(x)} g^p d\mu\right)^{1\over p}
\end{equation}
for every $f\in L^p_\mu(X)$ and every $p$-weak upper gradient $g\in L^p_\mu(X)$ for $f$. (For the definition of the concept of $p$-weak upper gradient, see Definition \ref{Def-p-weak-upper-gradient}.) As $\mu$ is doubling, for each $Q_{r}(\bar x)$ with $r>0$ and $\bar x\in X$ we have
\begin{equation}\label{Equa-Inject-compact}
{\mu(Q_{\rho}(x))\over\mu(Q_{r}(\bar x))}\geq 4^{-\DiM}\left({\rho\over r}\right)^\DiM
\end{equation}
for all $x\in Q_{r}(\bar x)$ and all $0<\rho\leq r$, where $\DiM:={\ln(C_d)\over\ln(2)}$ (see \cite[Lemma 4.7]{hajlasz02}).

From now on, we suppose $p>\DiM$ and we fix an integer $m\geq 1$.  

Let $\mathcal{O}(X)$ be the class of open subsets of $X$ and let $E:W^{1,p}_\mu(X;\RR^m)\times \mathcal{O}(X)\to[0,\infty]$ be the variational integral defined by 
\begin{equation}\label{Metric-Funct-1}
E(u,A):=\int_A L(x,\nabla_\mu u(x))d\mu(x),
\end{equation}
where $L:X\times\MM\to[0,\infty]$ is a Borel measurable integrand not necessarily convex with respect to $\xi\in\MM$, where $\MM$ denotes the space of all $m\times N$ matrices with $N\geq 1$ an integer.

The space $W^{1,p}_\mu(X;\RR^m)$ denotes the class of $p$-Cheeger-Sobolev functions from $X$ to $\RR^m$  and $\nabla_\mu u$ is the $\mu$-gradient of $u$ (see \S 3.1 for more details). 

Let $\overline{E}:W^{1,p}_\mu(X;\RR^m)\times\mathcal{O}(X)\to[0,\infty]$ be the ``relaxed" variational functional of the variational integral $E$ with respect to the strong convergence in $L^p_\mu(X;\RR^m)$, i.e., 
$$
\overline{E}(u,A):=\inf\left\{\liminf_{n\to\infty} E(u_n,A):u_n\stackrel{L^p_\mu}{\to }u\right\}.
$$
The object of the paper is to study the problem of finding an integral representation for $\overline{E}(\cdot,X)$ in the case where $L$ has not $p$-growth and can take infinite values.

\subsection{General growth and ru-usc condition} Let $G:X\times \MM\to[0,\infty]$ be a Borel measurable integrand which is $p$-coercive, i.e., there exists $c>0$ such that for every $x\in X$ and every $\xi\in\MM$, 
\begin{eqnarray}\label{p-coercivity}
G(x,\xi)\geq c|\xi|^p,
\end{eqnarray}
and for which there exists $r>0$ such that 
\begin{eqnarray}\label{growth-on-Gx}
\sup_{|\xi|\leq r}G(\cdot,\xi)\in L^1_\mu(X).
\end{eqnarray}
\begin{remark}
If $\sup_{|\xi|\leq r}G(\cdot,\xi)\in L^\infty_\mu(X)$ then \eqref{growth-on-Gx} is satisfied. In particular, this latter condition holds when $G$ depends only on $\xi$ and is convex. 
\end{remark}
We also assume that there exists $\gamma>0$ such that for every $x\in X$, every $t\in]0,1[$ and every $\xi,\zeta\in\MM$,
\begin{eqnarray}\label{hyp-convexity}
G(x,t\xi+(1-t)\zeta)\leq \gamma(1+G(x,\xi)+G(x,\zeta)).
\end{eqnarray}

\begin{remark}
If \eqref{hyp-convexity} holds and if $0\in{\rm int}\big(\{\xi\in\MM:G(\cdot,\xi)\in L^1_\mu(X)\}\big)$ then \eqref{growth-on-Gx} is verified, see \cite[Lemma 4.1]{oah-jpm12}.
\end{remark}

\begin{remark}\label{Remark-Dom-Gx-is-Convex}
If \eqref{hyp-convexity} holds then the effective domain $\GG_x$ of $G(x,\cdot)$ is convex.
\end{remark}

Let $\mathcal{G},\overline{\mathcal{G}}:W^{1,p}_\mu(X;\RR^m)\to[0,\infty]$ be the integral functionals defined by:
\begin{eqnarray}
&&\displaystyle\mathcal{G}(u):=\int_X G(x,\nabla_\mu u(x))d\mu(x);\label{Def-Of-LSC-G-For-ENdofTheProof-BiS}\\
&&\displaystyle\overline{\mathcal{G}}(u):=\inf\left\{\liminf_{n\to\infty} \mathcal{G}(u_n):u_n\stackrel{L^p_\mu}{\to }u\right\}.\label{Def-Of-LSC-G-For-ENdofTheProof}
\end{eqnarray}
Let us denote the effective domains of the functionals $\mathcal{G}$ and $\overline{\mathcal{G}}$ by $\mathfrak{G}$ and $\overline{\mathfrak{G}}$ respectively. We futhermore assume that:
\begin{eqnarray}
&&\hskip-20mm\mathfrak{G}=\overline{\mathfrak{G}};\label{lsc-intGx}\\
&&\hskip-20mm\hbox{if }u\in \mathfrak{G}\hbox{ then }\lim_{r\to 0}\mint_{Q_r(x)}\big|G(y,\nabla_\mu u(x))-G(x,\nabla_\mu u(x)\big|d\mu(y)=0\hbox{ for $\mu$-a.a. $x\in X$.}\label{growth-on-Gx-2}
\end{eqnarray}

\begin{remark}
The assumptions \eqref{growth-on-Gx-2} is satisfied in the following cases:
\begin{enumerate}[leftmargin=*]
\item if $G(\cdot,\xi)\in L^1_\mu(X)$ for all $\xi\in\MM$ then \eqref{growth-on-Gx-2} holds;
\item if $G$ only depends on $\xi$ then \eqref{growth-on-Gx-2} holds;
\item if $G(x,\xi)=G_1(x)+G_2(\xi)$ for all $x\in X$ and all $\xi\in \MM$ and if $G_1\in L^1_\mu(X)$ then \eqref{growth-on-Gx-2} holds;
\item if $G(x,\xi)=G_1(x)G_2(\xi)$ for all $x\in X$ and all $\xi\in \MM$ and if $G_1\in L^1_\mu(X)$ then \eqref{growth-on-Gx-2} holds
\end{enumerate}
\end{remark}

Throughout the paper, we assume that $L$ has $G$-growth, i.e.,  there exist $\alpha,\beta>0$ such that for every $x\in X$ and every $\xi\in\MM$,
\begin{eqnarray}\label{Gx-growth}
\alpha G(x,\xi)\leq L(x,\xi)\leq \beta (1+G(x,\xi)).
\end{eqnarray}

\begin{remark}\label{Remark-Dom-Lx-is-Convex}
If \eqref{hyp-convexity} and \eqref{Gx-growth} hold then the effective domain $\LL_x$ of $L(x,\cdot)$ is equal to $\GG_x$, and so is convex.
\end{remark}

\begin{remark}
If \eqref{Gx-growth} is satisfied then the effective domain of the functional $\overline{E}(\cdot,X)$ is equal to $\overline{\mathfrak{S}}$, and so to $\mathfrak{S}$ when \eqref{lsc-intGx} holds.
\end{remark}

When $G(x,\cdot)\equiv |\cdot |^p$ we say that $L$ has $p$-growth. The $p$-growth case was already studied in \cite{AHM15}. 
The object of this paper is to deal with the $G$-growth case.  For this, in addition, we need to suppose that $L$ is radially uniformly upper semicontinuous (ru-usc), i.e., there exists $a\in L^1_{\mu}(X;]0,\infty])$ such that
\begin{eqnarray}\label{ru-usc-condition}
\limsup_{t\to 1^-}\Delta^a_L(t)\leq 0
\end{eqnarray}
with $\Delta_L^a:[0,1]\to]-\infty,\infty]$ given by
$$
\Delta_{L}^a(t):=\sup_{x\in X}\sup_{\xi\in \LL_x}{L(x,t\xi)-L(x,\xi)\over a(x)+L(x,\xi)}.
$$
(For more details on the concept of ru-usc, see \S 3.3.)

\subsection{Integral representation theorem}

In what follows $p>\DiM$, where $\DiM:={\ln(C_d)\over \ln(2)}$ with $C_d\geq 1$ given by the inequality \eqref{D-meas}, and $m\geq 1$. Let $\mathcal{Q}_\mu L:X\times\MM\to[0,\infty]$, called the $\mu$-quasiconvexification of $L$, be given by
\begin{equation}\label{DefInITioN-Of-Hrho-mu-Lx-xi}
\mathcal{Q}_\mu L(x,\xi):=\limsup_{\rho\to 0}\inf\left\{\mint_{Q_\rho(x)}L(y,\xi+\nabla_\mu w(y))d\mu(y):w\in W^{1,p}_{\mu,0}(Q_\rho(x);\RR^m)\right\},
\end{equation}
where, for each $A\in\mathcal{O}(X)$, $W^{1,p}_{\mu,0}(A;\RR^m)$ is the closure of ${\rm Lip}_0(A;\RR^m)$ with respect to $W^{1,p}_\mu$-norm with ${\rm Lip}_0(A;\RR^m):=\big\{u\in{\rm Lip}( X;\RR^m):u=0\hbox{ on } X\setminus A\big\}$, where ${\rm Lip}( X;\RR^m)$ is the class of Lipschitz functions from $ X$ to $\RR^m$ (see \S 3.1 for more details). The main result of the paper is the following.

\begin{theorem}\label{MainTheorem}
If \eqref{p-coercivity}, \eqref{growth-on-Gx}, \eqref{hyp-convexity}, \eqref{lsc-intGx}, \eqref{growth-on-Gx-2}, \eqref{Gx-growth} and \eqref{ru-usc-condition} hold then   
\begin{equation}\label{Main-Relaxed-InTeGrAl}
\overline{E}(u,X)=\left\{
\begin{array}{ll}
\displaystyle \int_X \widehat{\mathcal{Q}_\mu L}(x,\nabla_\mu u(x))d\mu(x)&\hbox{if }u\in\mathfrak{G}\\
\infty&\hbox{if }u\in W^{1,p}_\mu( X;\RR^m)\setminus\mathfrak{G},
\end{array}
\right.
\end{equation}
where $\widehat{\mathcal{Q}_\mu L}:X\times \MM\to[0,\infty]$ is given by
$$
\widehat{\mathcal{Q}_\mu L}(x,\xi)=\liminf\limits_{t\to 1^-}\mathcal{Q}_\mu L(x,t\xi).
$$
\end{theorem}
 As a consequence of Theorem \ref{MainTheorem}, we have
\begin{corollary}\label{CoRoLLaRY---MainTheorem}
Under the hypotheses of Theorem {\rm\ref{MainTheorem}}, if $\mathcal{Q}_\mu G=G$ and if, for each $x\in X$, $\mathcal{Q}_\mu L(x,\cdot)$ is lsc on the interior ${\rm int}(\LL_x)$ of the effective domain $\LL_x$ of $L(x,\cdot)$, then \eqref{Main-Relaxed-InTeGrAl} holds and
\begin{equation}\label{Main-Relaxed-FOrMUlA}
\widehat{\mathcal{Q}_\mu L}(x,\xi)=\left\{
\begin{array}{ll}
\mathcal{Q}_\mu L(x,\xi)&\hbox{if }x\in X\hbox{ and }\xi\in{\rm int}(\LL_x)\\
\lim\limits_{t\to 1^-}\mathcal{Q}_\mu L(x,t\xi)&\hbox{if }x\in X\hbox{ and }\xi\in\partial\LL_x\\
\infty&\hbox{otherwise.}
\end{array}
\right.
\end{equation}
\end{corollary}
\begin{proof}
As $\mathcal{Q}_\mu G=G$ we have $\mathcal{Q}_\mu\LL_x=\LL_x$ for all $x\in X$, where $\mathcal{Q}\LL_x$ denotes the effective domain of $\mathcal{Q}_\mu L(x,\cdot)$. Moreover, from \eqref{hyp-convexity} is is easily seen that $\LL_x$ is convex for all $x\in X$, hence \eqref{Homothecie-Assumption-Bis} holds, and \eqref{Main-Relaxed-FOrMUlA} follows from Theorem \ref{Extension-Result-for-ru-usc-Functions}.
\end{proof}


\section{Auxiliary results}

\subsection{The $p$-Cheeger-Sobolev spaces} Let $p>1$ be a real number, let $(X,d,\mu)$ be a metric measure space, where $(X,d)$ is a length space which is separable and compact, and $\mu$ is a positive Radon measure on $X$. We begin with the concept of upper gradient introduced by Heinonen and Koskela (see \cite{heinonen-koskela98}). 
\begin{definition}
A Borel function $g: X\to[0,\infty]$ is said to be an upper gradient for $f: X\to\RR$  if 
$
|f(c(1))-f(c(0))|\leq\int_0^1 g(c(s))ds
$
for all continuous rectifiable curves $c:[0,1]\to  X$. 
\end{definition}
The concept of upper gradient has been generalized by Cheeger as follows (see \cite[Definition 2.8]{cheeger99}). 
\begin{definition}\label{Def-p-weak-upper-gradient}
A function $g\in L^p_\mu( X)$ is said to be a $p$-weak upper gradient for $f\in L^p_\mu( X)$ if there exist $\{f_n\}_n\subset L^p_\mu( X)$ and $\{g_n\}_n\subset L^p_\mu( X)$ such that for each $n\geq 1$, $g_n$ is an upper gradient for $f_n$, $f_n\to f$ in $L^p_\mu( X)$ and $g_n\to g$ in $L^p_\mu( X)$. 
\end{definition}
Denote the algebra of Lipschitz functions from $ X$ to $\RR$ by ${\rm Lip}( X)$. (As $ X$ is compact, every Lipschitz function from $ X$ to $\RR$ is bounded.) From Cheeger and Keith (see \cite[Theorem 4.38]{cheeger99} and \cite[Definition 2.1.1 and Theorem 2.3.1]{keith1-04}) we have

\begin{theorem}\label{cheeger-theorem}
If $\mu$ is doubling, i.e., \eqref{D-meas} holds, and $ X$ supports a weak $(1,p)$-Poincar\'e inequality, i.e., $\eqref{1-p-PI}$ holds, then there exists a countable family $\{( X_k,\xi^k)\}_k$ of $\mu$-measurable disjoint subsets $ X_k$ of $ X$ with $\mu( X\setminus\cup_k X_k)=0$ and of functions $\xi^k=(\xi^k_1,\cdots,\xi^k_{N(k)}): X\to\RR^{N(k)}$ with $\xi^k_i\in{\rm Lip}( X)$ satisfying the following properties{\rm:}
\begin{enumerate}[leftmargin=*]
\item[\rm(a)]  there exists an integer $N\geq 1$ such that $N(k)\in\{1,\cdots, N\}$ for all $k;$ 
\item[\rm(b)] for every $k$ and every $f\in{\rm Lip}( X)$ there is a unique $D_\mu^k f\in L^\infty_\mu( X_k;\RR^{N(k)})$ such that for $\mu$-a.e. $x\in  X_k$,
$$
\lim_{\rho\to 0}{1\over\rho}\|f-f_x\|_{L^\infty_{\mu}(Q_\rho(x))}=0,
$$
where $f_x\in {\rm Lip}( X)$ is given by $f_x(y):=f(x)+D_\mu^k f(x)\cdot(\xi^k(y)-\xi^k(x));$ in particular 
$$
D_\mu^k f_x(y)=D_\mu^k f(x)\hbox{ for $\mu$-a.a. $y\in  X_k$};
$$
\item[\rm(c)] the operator $D_\mu:{\rm Lip}( X)\to L^\infty_\mu( X;\RR^N)$ given by
$$
D_\mu f:=\sum_k \mathds{1}_{X_k}D_\mu^k f,
$$
where $\mathds{1}_{ X_k}$ denotes the characteristic function of $ X_k$, is linear and, for each $f,g\in{\rm Lip}( X)$, one has 
$$
D_\mu(fg)=fD_\mu g+gD_\mu f;
$$
\item[\rm(d)] for every $f\in{\rm Lip}( X)$, $D_\mu f=0$ $\mu$-a.e. on every $\mu$-measurable set where $f$ is constant.
\end{enumerate}
\end{theorem}

\begin{remark}
Theorem \ref{cheeger-theorem} is true without the assumption that $(X,d)$ is a length space.
\end{remark}

Let ${\rm Lip}( X;\RR^m):=[{\rm Lip( X)}]^m$ and let $\nabla_\mu:{\rm Lip}( X;\RR^m)\to L^\infty_\mu( X;\MM)$ given by 
$$
\nabla_\mu u:=\left(
\begin{array}{c}
D_\mu u_1\\
\vdots\\ 
D_\mu u_m
\end{array}
\right)
\hbox{ with }u=(u_1,\cdots,u_m).
$$

From Theorem \ref{cheeger-theorem}(c) we see that for every $u\in {\rm Lip}( X;\RR^m)$ and every $f\in {\rm Lip}( X)$, one has
\begin{equation}\label{Mu-der-Prod}
\nabla_\mu (fu)=f\nabla_\mu u+D_\mu f\otimes u.
\end{equation}
\begin{definition} 
The $p$-Cheeger-Sobolev space $W^{1,p}_\mu( X;\RR^m)$  is defined as the completion of ${\rm Lip}( X;\RR^m)$ with respect to the norm
\begin{equation}\label{W1pmu-norm}
\|u\|_{W^{1,p}_\mu( X;\RR^m)}:=\|u\|_{L^p_\mu( X;\RR^m)}+\|\nabla_\mu u\|_{L_\mu^p( X;\MM)}.
\end{equation}
\end{definition}
Taking Proposition \ref{Fundamental-Proposition-for-CalcVar-in-MMS}(a) below into account, since $\|\nabla_\mu u\|_{L_\mu^p(X;\MM)}\leq \|u\|_{W^{1,p}_\mu(X;\RR^m)}$ for all $u\in{\rm Lip}(X;\RR^m)$ the linear map $\nabla_\mu$ from ${\rm Lip}(X;\RR^m)$ to $L_\mu^p(X;\MM)$  has a unique extension to $W^{1,p}_\mu(X;\RR^m)$ which will still be denoted by $\nabla_\mu$ and will be called the $\mu$-gradient.

\begin{remark}
When $ X$ is a the closure of a bounded open subset $\Omega$ of $\RR^N$ and $\mu$ is the Lebesgue measure on $\overline{\Omega}$, we retreive the (classical) Sobolev spaces $W^{1,p}( \Omega;\RR^m)$.  If $X$ is a compact manifold $M$ and if $\mu$ is the superficial measure on $M$, we obtain the (classical) Sobolev spaces $W^{1,p}(M;\RR^m)$ on the compact manifold $M$. For more details on the various possible extensions of the classical theory of the Sobolev spaces to the setting of metric measure spaces, we refer to \cite[\S 10-14]{heinonen07} (see also \cite{cheeger99, shanmugalingam00, gol-tro01,hajlasz02}).
\end{remark}

The following proposition (whose proof is given below, see also \cite[Proposition 2.28]{AHM15}) provides useful properties for dealing with calculus of variations in the metric measure setting.

\begin{proposition}\label{Fundamental-Proposition-for-CalcVar-in-MMS}
Under the hypotheses of Theorem {\rm\ref{cheeger-theorem}}, we have{\rm:}
\begin{enumerate}[leftmargin=*]
\item[\rm(a)] the $\mu$-gradient is closable in $W^{1,p}_\mu( X;\RR^m)$, i.e., for every $u\in W^{1,p}_\mu( X;\RR^m)$ and every $A\in\mathcal{O}( X)$, if $u(x)=0$ for $\mu$-a.a. $x\in A$ then $\nabla_\mu u(x)=0$ for $\mu$-a.a. $x\in A;$
\item[\rm(b)] $ X$ supports a $p$-Sobolev inequality, i.e., there exist $C_S>0$ and $\chi\geq 1$ such that
\begin{equation}\label{Poincare-Inequality}
\left(\int_{Q_\rho(x)}|v|^{\chi p}d\mu\right)^{1\over\chi p}\leq \rho C_S\left(\int_{Q_\rho(x)}|\nabla_\mu v|^pd\mu\right)^{1\over p}
\end{equation}
for all $0<\rho\leq \rho_0$, with $\rho_0>0$, and all $v\in W^{1,p}_{\mu,0}(Q_\rho(x);\RR^m)$, where, for each $A\in\mathcal{O}( X)$, $W^{1,p}_{\mu,0}(A;\RR^m)$ is the closure of ${\rm Lip}_0(A;\RR^m)$ with respect to $W^{1,p}_\mu$-norm defined in \eqref{W1pmu-norm} with 
$$
{\rm Lip}_0(A;\RR^m):=\big\{u\in{\rm Lip}( X;\RR^m):u=0\hbox{ on } X\setminus A\big\};
$$
\item[\rm(c)] $ X$ satisfies the Vitali covering theorem, i.e., for every $A\subset  X$ and every family $\mathcal{F}$ of closed balls in $ X$, if $\inf\{\rho>0:\overline{Q}_\rho(x)\in\mathcal{F}\}=0$ for all $x\in A$ then there exists a countable disjointed subfamily $\mathcal{G}$ of $\mathcal{F}$ such that $\mu(A\setminus \cup_{Q\in\mathcal{G}}Q)=0;$ in other words, $A\subset \big(\cup_{Q\in\mathcal{G}}Q\big)\cup N$ with $\mu(N)=0;$
\item[\rm(d)] for every $u\in W^{1,p}_\mu( X;\RR^m)$ and $\mu$-a.e. $x\in  X$ there exists $u_x\in W^{1,p}_\mu( X;\RR^m)$ such that{\rm:}
\begin{eqnarray}
&&\nabla_\mu u_x(y)=\nabla_\mu u(x)\hbox{ for $\mu$-a.a. $y\in  X$};\label{FinALAssuMpTIOnOne}\\
&&\displaystyle\lim_{\rho\to 0}{1\over\rho}\|u-u_x\|_{L^\infty_{\mu}(Q_\rho(x);\RR^m)}=0\hbox{ if }p>\DiM,\label{FinALAssuMpTIOnTwo}
\end{eqnarray}
where $\DiM:={\ln(C_d)\over \ln(2)}$ with $C_d\geq 1$ given by the inequality \eqref{D-meas}{\rm;}
\item[\rm(e)] for every $x\in  X$, every $\rho>0$ and every $\tau\in]0,1[$ there exists a Uryshon function $\varphi\in{\rm Lip}( X)$ for the pair $( X\setminus Q_\rho(x),\overline{Q}_{\tau\rho}(x))\footnote{Given a metric space $( X,d)$, by a Uryshon function from $ X$ to $\RR$ for the pair $( X\setminus V,K)$, where $K\subset V\subset X$ with $K$ compact and $V$ open, we mean a continuous function $\varphi: X\to\RR$ such that $\varphi(x)\in[0,1]$ for all $x\in X$, $\varphi(x)=0$ for all $x\in X\setminus V$ and $\varphi(x)=1$ for all $x\in K$.}$ such that 
$$
\|D_\mu\varphi\|_{L^\infty_\mu( X;\RR^N)}\leq{\theta\over\rho(1-\tau)}
$$
for some $\theta>0;$
\item[\rm(f)] for $\mu$-a.e. $x\in  X$,
\begin{equation}\label{DoublINgAssUMpTiON}
\lim_{\tau\to 1^-}\liminf_{\rho\to0}{\mu(Q_{\tau\rho}(x))\over\mu(Q_\rho(x))}=\lim_{\tau\to 1^-}\limsup_{\rho\to0}{\mu(Q_{\tau\rho}(x))\over\mu(Q_\rho(x))}=1.
\end{equation}
\end{enumerate}
\end{proposition}

\begin{remark}\label{ReMArK-VItALi-For-OpEN-SEtS}
As $\mu$ is a Radon measure, if $ X$ satisfies the Vitali covering theorem, i.e., Proposition \ref{Fundamental-Proposition-for-CalcVar-in-MMS}(c) holds, then for every $A\in\mathcal{O}( X)$ and every $\eps>0$ there exists a countable family $\{Q_{\rho_i}(x_i)\}_{i\in I}$ of disjoint open balls of $A$ with $x_i\in A$, $\rho_i\in]0,\eps[$ and $\mu(\partial Q_{\rho_i}(x_i))=0$ such that $\mu\big(A\setminus\cup_{i\in I}Q_{\rho_i}(x_i)\big)=0$.
\end{remark}

\begin{proof}[\bf Proof of Proposition \ref{Fundamental-Proposition-for-CalcVar-in-MMS}]
Firstly, $ X$ satisfies the Vitali covering theorem, i.e., the property (c) holds, because $\mu$ is doubling (see \cite[Theorem 2.8.18]{Fed-69}). Secondly, the closability of the $\mu$-gradient in ${\rm Lip}( X;\RR^m)$, given by Theorem \ref{cheeger-theorem}(d), can be extended from ${\rm Lip}( X;\RR^m)$ to $W^{1,p}_\mu( X;\RR^m)$ by using the closability theorem of Franchi, Haj{\l}asz and Koskela (see \cite[Theorem 10]{fran-haj-kos99}). Thus, the property (a) is satisfied. Thirdly,  according to Cheeger (see \cite[\S 4, p. 450]{cheeger99} and also \cite{Haj-Kos-95,hajlaszkoskela00}), since $\mu$ is doubling and $ X$ supports a weak $(1,p)$-Poincar\'e inequality, we can assert that there exist $c>0$ and $\chi>1$ such that for every $0<\rho\leq\rho_0$, with $\rho_0\geq 0$, every $v\in W^{1,p}_{\mu,0}( X;\RR^m)$ and every $p$-weak upper gradient $g\in L^p_\mu( X;\RR^m)$ for $v$,
\begin{equation}\label{Cheeger-Coro-Eq1}
\left(\int_{Q_\rho(x)}|v|^{\chi p}d\mu\right)^{1\over\chi p}\leq \rho c\left(\int_{Q_\rho(x)}|g|^pd\mu\right)^{1\over p}.
\end{equation}
On the other hand, from Cheeger (see \cite[Theorems 2.10 and 2.18]{cheeger99}), for each $w\in W^{1,p}_\mu( X)$ there exists a unique $p$-weak upper gradient for $w$, denoted by $g_w\in L ^p_\mu( X)$ and called the minimal $p$-weak upper gradient for $w$, such that for every  $p$-weak upper gradient $g\in L^p_\mu( X)$ for $w$, $g_w(x)\leq g(x)$ for $\mu$-a.a.  $x\in  X$. Moreover (see \cite[\S 4]{cheeger99} and also \cite[\S B.2, p. 363]{bjorn-bjorn-11}, \cite{bjorn00} and \cite[Remark 2.15]{gong-haj-13}), there exists $\theta\geq 1$ such that for every $w\in W^{1,p}_\mu( X)$ and $\mu$-a.e. $x\in  X$, 
\begin{equation}\label{Equa-Inject-compact-1}
{1\over\theta} g_w(x)\leq|D_\mu w(x)|\leq\theta g_w(x).
\end{equation}
As for $v=(v_i)_{i=1,\cdots,m}\in W^{1,p}_\mu( X;\RR^m)$ we have $\nabla_\mu v=(D_\mu v_i)_{i=1,\cdots,m}$, it follows that 
\begin{equation}\label{Cheeger-Coro-Eq2}
{1\over \theta} |g_v(x)|\leq|\nabla_\mu v(x)|\leq\theta|g_v(x)|
\end{equation}
for $\mu$-a.a. $x\in  X$, where $g_v:=(g_{v_i})_{i=1,\cdots,m}$ is naturally called the minimal $p$-weak upper gradient for $v$. Combining \eqref{Cheeger-Coro-Eq1} with \eqref{Cheeger-Coro-Eq2} we obtain the property (b). Fourthly,  from Bj{\"o}rn (see \cite[Corollary 4.6(ii)]{bjorn00} we see that for every $k$, every $u\in W^{1,p}_\mu( X;\RR^m)$ and $\mu$-a.e. $x\in  X_k$,
$$
\nabla_\mu u_x(y)=\nabla_\mu u(x)\hbox{ for }\mu\hbox{-a.a. }y\in  X_k,
$$
where $u_x\in W^{1,p}_\mu( X;\RR^m)$ is given by 
$$
u_x(y):=u(y)-u(x)-\nabla_\mu u(x)\cdot(\xi^k(y)-\xi^k(x)),
$$
 and if $p>\DiM$ then $u$ is $L^\infty_\mu$-differentiable at $x$, i.e.,
$$
\lim_{\rho\to 0}{1\over\rho}\|u(y)-u_x(y)\|_{L^\infty_\mu(Q_\rho(x);\RR^m)}=0.
$$
Hence the property (d) is verified. Fifthly, given $\rho>0$, $\tau\in]0,1[$ and $x\in  X$, there exists a Uryshon function $\varphi\in {\rm Lip}( X)$ for the pair $( X\setminus Q_\rho(x)), \overline{Q}_{\tau\rho}(x))$ such 
$$
\|{\rm Lip}\varphi\|_{L^\infty_\mu( X)}\leq{1\over\rho(1-\tau)},
$$
where for every $y\in  X$,
$$
{\rm Lip}\varphi(y):=\limsup_{d(y,z)\to0}{|\varphi(y)-\varphi(z)|\over d(y,z)}.
$$
But, since $\mu$ is doubling and $ X$ supports a weak $(1,p)$-Poincar\'e inequality, from Cheeger (see \cite[Theorem 6.1]{cheeger99}) we have ${\rm Lip}\varphi(y)=g_\varphi(y)$ for $\mu$-a.a. $y\in  X$, where $g_{\varphi}$ is the minimal $p$-weak upper gradient for $\varphi$. Hence 
$$
\|D_\mu\varphi\|_{L^\infty_\mu( X;\RR^N)}\leq{\theta\over\rho(1-\tau)}
$$ 
because $|D_\mu\varphi(y)|\leq \theta|g_{\varphi}(y)|$ for $\mu$-a.a. $y\in  X$. Consequently the property (e) holds. Finally, as $(X,d)$ is a length space, from Colding and Minicozzi II (see \cite{colding-minicozzi98} and \cite[Proposition 6.12]{cheeger99}) we can assert that there exists $\delta>0$ such that for every $x\in  X$, every $\rho>0$ and every $\tau\in]0,1[$, 
$$
\mu(Q_\rho(x)\setminus Q_{\tau\rho}(x))\leq 2^\delta(1-\tau)^\delta\mu(Q_\rho(x)),
$$
which implies the property (f).
\end{proof}

\medskip

The following result comes from Haj{\l}asz and Koskela (see \cite[Theorem 5.1]{hajlaszkoskela00} and \cite[Theorem 9.7 and Remark 9.8]{hajlasz02}). 

\begin{theorem}\label{Embedding-TheO-1}
Assume that the assumptions of Theorem {\rm\ref{cheeger-theorem}} hold and let $\DiM:={\ln(C_d)\over \ln(2)}$ where $C_d\geq 1$ is given by the inequality \eqref{D-meas}. If $p>\DiM$ then for every $r>0$ and every $\bar x\in X$ there exists $C(r,\bar x)>0$ such that 
$$
|u(y)-u(z)|\leq C(r,\bar x)d(y,z)^{1-{\DiM\over p}}\left(\int_{Q_r(\bar x)}|\nabla_\mu u|^pd\mu\right)^{1\over p}
$$
for all $u\in W^{1,p}_\mu( X;\RR^m)$ and all $y,z\in Q_r(\bar x)$. 
\end{theorem}

\begin{proof}[\bf Proof of Theorem \ref{Embedding-TheO-1}]
Since $(X,d,\mu)$ is a compact (and so complete) doubling metric space, $(X,d,\mu)$ is proper, i.e., every closed ball is compact (see \cite[Lemma 4.1.14]{HKST-book-2015}). Moreover, by assumption, $(X,d)$ is a length space, i.e., the distance between any two points equals infimum of lengths of curves connecting the points.  So, from \cite[Theorem 9.7 and Remark 9.8]{hajlasz02} we can assert that there exists $c>0$ such that
\begin{equation}\label{Embedding-TheO-1-EquAtION1}
|w(y)-w(z)|\leq c r^{\DiM\over p}d(y,z)^{1-{\DiM\over p}}\left(\mint_{Q_r(\bar x)}g_w^pd\mu\right)^{1\over p}
\end{equation}
for all $w\in W^{1,p}_\mu( X)$, all $\bar x\in X$, all $r>0$ and all $y,z\in Q_r(\bar x)$, where $g_w\in L^p_\mu(X)$ denotes the minimal $p$-weak upper gradient for $w$.
On the other hand, from \eqref{Equa-Inject-compact} it is easy to see that for every $r>0$ and every $\bar x\in X$  there exists $\theta(r,\bar x)>0$ such that
$$
\mu(Q_{r}(\bar x))\geq \theta(r,\bar x)r^\DiM.
$$
Moreover, by the left inequality in \eqref{Equa-Inject-compact-1} we have $\mmint_{Q_\rho(x)}g_w^pd\mu\leq\alpha^p\mmint_{Q_\rho(x)}|D_\mu w|^pd\mu$. Thus, for each $r>0$, each $\bar x\in X$ and each $y,z\in Q_r(\bar x)$, \eqref{Embedding-TheO-1-EquAtION1} can be rewritten as follows 
$$
|w(y)-w(z)|\leq C(r,\bar x)d(y,z)^{1-{\DiM\over p}}\left(\int_{Q_r(\bar x)}|D_\mu w|^pd\mu\right)^{1\over p}
$$
with $C(r,\bar x)={c\alpha\over\theta(r,\bar x)}>0$. It follows that for every $r>0$ and every $\bar x\in X$, we have 
\begin{eqnarray*}
|u(y)-u(z)|&\leq& C(r,\bar x)d(y,z)^{1-{\DiM\over p}}\sum_{i=1}^m\left(\int_{Q_r(\bar x)}|D_\mu u_i|^pd\mu\right)^{1\over p}\\
&\leq&C(r,\bar x)d(y,z)^{1-{\DiM\over p}}\left(\int_{Q_r(\bar x)}\sum_{i=1}^m|D_\mu u_i|^pd\mu\right)^{1\over p}\\
&=&C(r,\bar x)d(y,z)^{1-{\DiM\over p}}\left(\int_{Q_r(\bar x)}|\nabla_\mu u|^pd\mu\right)^{1\over p}
\end{eqnarray*}
for all $u\in W^{1,p}_\mu(X;\RR^m)$ and all $y,z\in Q_r(\bar x)$, and the proof is complete.
\end{proof}

\medskip

Denote the space of continuous functions from $X$ to $\RR^m$ by $C(X;\RR^m)$. As a consequence of Theorem \ref{Embedding-TheO-1} we have  

\begin{corollary}\label{L-infty-CSE-1}
Under the assumptions of Theorem {\rm\ref{cheeger-theorem}}, if $p>\DiM$ then $W^{1,p}_\mu(X;\RR^m)$ continuously embeds into $C(X;\RR^m)$, i.e.,  $W^{1,p}_\mu(X;\RR^m)\subset C(X;\RR^m)$ and  there exists $K_0>0$ such that 
\begin{equation}\label{L-infty-CSE-1-GoAl}
\|u\|_{C(X;\RR^m)}\leq K_0\|u\|_{W^{1,p}_\mu(X;\RR^m)}
\end{equation}
for all $u\in W^{1,p}_\mu(X;\RR^m)$. Moreover, there exists $K_1>0$ such that
\begin{equation}\label{L-infty-CSE-1-EQUA0}
|u(y)-u(z)|\leq K_1d(y,z)^{1-{\DiM\over p}}\|\nabla_\mu u\|_{L^p_\mu(X;\MM)}
\end{equation}
for all $u\in W^{1,p}_\mu(X;\RR^m)$ and all $y,z\in X$.
\end{corollary}

\begin{proof}[\bf Proof of Corollary \ref{L-infty-CSE-1}]
Applying Theorem \ref{Embedding-TheO-1} with $r={\rm diam}(X)$ and for a fixed $\bar x=x_0\in X$, where ${\rm diam}(X)=\sup\{d(y,z):y,z\in X\}<\infty$ because $(X,d)$ is compact, we see that
\begin{eqnarray}
|u(y)-u(z)|&\leq& C\left({\rm diam}(X),x_0\right)d(y,z)^{1-{\DiM\over p}}\|\nabla_\mu u\|_{L^p_\mu(X;\MM)}\nonumber\\
&\leq&C\left({\rm diam}(X),x_0\right){\rm diam}(X)^{1-{\DiM\over p}}\|\nabla_\mu u\|_{L^p_\mu(X;\MM)}\label{L-infty-CSE-1-EQUA1}
\end{eqnarray}
for all $u\in W^{1,p}_\mu(X;\RR^m)$ and all $y,z\in X$. Hence \eqref{L-infty-CSE-1-EQUA0} holds with $K_1=C\left({\rm diam}(X),x_0\right)$ and every $u\in W^{1,p}_\mu(X;\RR^m)$ is $(1-{\DiM\over p})$-H\"older continuous. In particular, it follows that $W^{1,p}_\mu(X;\RR^m)\subset C(X;\RR^m)$. On the other hand, given any $u\in W^{1,p}_\mu(X;\RR^m)$ and any $y\in X$, we have $|u(y)|^p\leq 2^p\left(|u(y)-u(z)|^p+|u(z)|^p\right)$ for all $z\in X$, and consequently
\begin{equation}\label{L-infty-CSE-1-EQUA2}
\mu(X)^{1\over p}|u(y)|\leq 2^{1+{1\over p}}\left(\int_X |u(y)-u(z)|^pd\mu(z)\right)^{1\over p}+2^{1+{1\over p}}\|u\|_{L^p_\mu(X;\RR^m)}.
\end{equation}
But, by \eqref{L-infty-CSE-1-EQUA1} we have 
\begin{equation}\label{L-infty-CSE-1-EQUA3}
\left(\int_X |u(y)-u(z)|^pd\mu(z)\right)^{1\over p}\leq \mu(X)^{1\over p}C\left({\rm diam}(X),x_0\right){\rm diam}(X)^{1-{\DiM\over p}}\|\nabla_\mu u\|_{L^p_\mu(X;\MM)}.
\end{equation}
Hence, combining \eqref{L-infty-CSE-1-EQUA2} and \eqref{L-infty-CSE-1-EQUA3} we deduce that for every $y\in X$,
\begin{eqnarray*}
|u(y)|&\leq& 2^{1+{1\over p}}C\left({\rm diam}(X),x_0\right){\rm diam}(X)^{1-{\DiM\over p}}\|\nabla_\mu u\|_{L^p_\mu(X;\MM)}+{2^{1+{1\over p}}\over\mu(X)^{1\over p}}\|u\|_{L^p_\mu(X;\RR^m)}\\
&\leq&K_0\|u\|_{W^{1,p}_\mu(X;\RR^m)}
\end{eqnarray*}
with $K_0=\sup\left\{2^{1+{1\over p}}C\left({\rm diam}(X),x_0\right){\rm diam}(X)^{1-{\DiM\over p}},{2^{1+{1\over p}}\over\mu(X)^{1\over p}}\right\}$, and \eqref{L-infty-CSE-1-GoAl} follows.
\end{proof}

\medskip

As a consequence of Corollary \ref{L-infty-CSE-1} we have the following $L^\infty_\mu$-compactness result in the framework of the $p$-Cheeger-Sobolev spaces with $p>\DiM$.

\begin{corollary}\label{L-infty-CSE}
Under the assumptions of Theorem {\rm\ref{cheeger-theorem}}, if $p>\DiM$ and if $u\in W^{1,p}_\mu( X;\RR^m)$ and $\{u_n\}_n\subset W^{1,p}_\mu( X;\RR^m)$ are such that 
\begin{equation}\label{L-infty-CSE-HyPOthesis1}
\lim_{n\to\infty}\|u_n-u\|_{L^p_\mu(X;\RR^m)}=0\hbox{ and }\sup_{n\geq 1}\|\nabla_\mu u_n\|_{L^p_\mu(X;\MM)}<\infty, 
\end{equation}
then, up to a subsequence, 
\begin{equation}\label{L-infty-CSE-GoAL}
\lim_{n\to\infty}\|u_n-u\|_{L^\infty_\mu(X;\RR^m)}=0.
\end{equation}
\end{corollary}

\begin{proof}[\bf Proof of Corollary \ref{L-infty-CSE}]
From \eqref{L-infty-CSE-HyPOthesis1} we deduce that $\sup_{n\geq 1}\|u_n\|_{W^{1,p}_\mu(X;\RR^m)}<\infty$, and using Corollary \ref{L-infty-CSE-1} we can assert that $\sup_{n\geq 1}\|u_n\|_{C(X;\RR^m)}<\infty$, i.e., $\{u_n\}_n$ is bounded in $C(X;\RR^m)$ with $(X,d)$ a compact metric space. On the other hand, using \eqref{L-infty-CSE-1-EQUA0} it is easy to see that $\{u_n\}_n$ is equicontinuous, and \eqref{L-infty-CSE-GoAL} follows from Arzel\`a-Ascoli's theorem.
\end{proof}

\subsection{The De Giorgi-Letta lemma}

Let $ X=( X,d)$ be a metric space, let $\mathcal{O}( X)$ be the class of open subsets of $ X$ and let $\mathcal{B}( X)$ be the class of Borel subsets of $ X$, i.e., the smallest $\sigma$-algebra containing the open (or equivalently the closed) subsets of $ X$. The following result is due to De Giorgi and Letta (see \cite{degiorgi-letta77} and also \cite[Lemma 3.3.6 p. 105]{buttazzo89}).

\begin{lemma}\label{DeGiorgi-Letta-Lemma}
Let $\mathcal{S}:\mathcal{O}( X)\to[0,\infty]$ be an increasing set function, i.e., $\mathcal{S}(A)\leq \mathcal{S}(B)$ for all $A,B\in\mathcal{O}( X)$ such $A\subset B$, satisfying the following four conditions{\rm:}
\begin{enumerate}[leftmargin=*]
\item[{\rm(a)}] $\mathcal{S}(\emptyset)=0;$
\item[{\rm(b)}] $\mathcal{S}$ is superadditive, i.e., $\mathcal{S}(A\cup B)\geq \mathcal{S}(A)+\mathcal{S}(B)$ for all $A,B\in\mathcal{O}( X)$ such that $A\cap B=\emptyset;$
\item[{\rm(c)}] $\mathcal{S}$ is subadditive, i.e., $\mathcal{S}(A\cup B)\leq \mathcal{S}(A)+\mathcal{S}(B)$ for all $A,B\in\mathcal{O}( X);$
\item[{\rm(d)}] there exists a finite Radon measure $\nu$ on $ X$ such that $\mathcal{S}(A)\leq\nu(A)$ for all $A\in\mathcal{O}( X)$.
\end{enumerate}
Then, $\mathcal{S}$ can be uniquely extended to a finite positive Radon measure on $ X$ which is absolutely continuous with respect to $\nu$.
\end{lemma}


\subsection{Integral representation of the Vitali envelope of a set function}
What follows was first developed in \cite{boufonmas98,bouchitte-bellieud00} (see also \cite{AHM15-PP}). Here we only give what you need for proving Theorem \ref{MainTheorem}. Let $(X,d)$ be a metric space, let $\mathcal{O}(X)$ be the class of open subsets of $X$ and let $\mu$ be a positive finite Radon measure on $X$.  We begin with the concept of differentiability with respect to $\mu$ of a set function.

\begin{definition}
We say that a set function $\Theta:\mathcal{O}(X)\to\RR$ is differentiable with respect to $\mu$ if
\begin{equation}\label{Vitali-Envelope-EQ1}
d_\mu \Theta(x):=\lim_{\rho\to 0}{\Theta(Q_\rho(x))\over\mu(Q_\rho(x))}
\end{equation}
exists and is finite  for $\mu$-a.e. $x\in X$. 
\end{definition}

\begin{remark}
It is easy to see that the limit in \eqref{Vitali-Envelope-EQ1} exists and is finite if and only if $-\infty<d_\mu^+ \Theta\leq d_\mu^-\Theta<\infty$, where $d^-_\mu\Theta:X\to[-\infty,\infty[$ and $d_\mu^+\Theta:X\to]-\infty,\infty]$ are given by:
\begin{eqnarray}
&&\hskip-14mm\displaystyle d^-_\mu\Theta(x):=\lim_{\rho\to 0}d^-_\mu\Theta(x,\rho)\hbox{ with }d^-_\mu\Theta(x,\rho):=\inf\left\{{\Theta(Q)\over\mu(Q)}: Q\in {\rm Ba}(X,x,\rho)\right\};\label{Lower-Differential-Measure}\\
&&\hskip-14mm\displaystyle d^+_\mu\Theta(x):=\lim_{\rho\to 0}d^+_\mu\Theta(x,\rho)\hbox{ with }d^+_\mu\Theta(x,\rho):=\sup\left\{{\Theta(Q)\over\mu(Q)}: Q\in {\rm Ba}(X,x,\rho)\right\},\label{Upper-Differential-Measure}
\end{eqnarray}
where ${\rm Ba}(X,x,\rho)$ denotes the class of open balls $Q$ of $X$ such that $x\in Q$, ${\rm diam}(Q)\in]0,\rho[$ and $\mu(\partial Q)=0$, where  $\partial Q:=\overline{Q}\setminus Q$. We then have $d_\mu\Theta=d^-_\mu\Theta=d^+_\mu\Theta$.
\end{remark}

\begin{remark}\label{Remark2-ViTaLi-ENVELopE}
In \eqref{Lower-Differential-Measure} and \eqref{Upper-Differential-Measure} we can replace ${\rm Ba}(X,x,\rho)$ by ${\rm Ba}(A,x,\rho)$ whenever $A\in\mathcal{O}(X)$ and $x\in A$.
\end{remark}


For each $\eps>0$ and each $A\in\mathcal{O}(X)$, we denote the class of countable families $\{Q_i:=Q_{\rho_i}(x_i)\}_{i\in I}$ of disjoint open balls of $A$ with $x_i\in A$, $\rho_i={\rm diam}(Q_i)\in]0,\eps[$ and $\mu(\partial Q_i)=0$ such that $\mu(A\setminus\cup_{i\in I}Q_i)=0$ by $\mathcal{V}_\eps(A)$.

\begin{definition}\label{VitaliEnvelopeDef}
Given $\Theta:\mathcal{O}(X)\to\RR$, for each $\eps>0$ we define $\Theta^\eps:\mathcal{O}(X)\to[-\infty,\infty]$ by
\begin{equation}\label{Pre-Vitali-Envelope-DEF}
{\Theta}^\eps(A):=\inf\left\{\sum_{i\in I}\Theta(Q_i):\{Q_i\}_{i\in I}\in \mathcal{V}_\eps(A)\right\}.
\end{equation}
By the Vitali envelope of $\Theta$ we call the set function $\Theta^*:\mathcal{O}(X)\to[-\infty,\infty]$ defined by
\begin{equation}\label{Vitali-Envelope-DEF}
\Theta^*(A):=\sup_{\eps>0}\Theta^\eps(A)=\lim_{\eps\to0}\Theta^\eps(A).
\end{equation}
\end{definition}

The interest of Definition \ref{VitaliEnvelopeDef} comes from the following integral representation result whose proof is given below. 

\begin{theorem}\label{Vitali-Envelope-Prop2}
Let $\Theta:\mathcal{O}(X)\to\RR$ be a set function satisfying the following two conditions{\rm:}
\begin{enumerate}[leftmargin=*]
\item[{\rm (a)}] there exists a finite Radon measure $\nu$ on $X$ which is absolutely continuous with respect to $\mu$ such that $|\Theta(A)|\leq\nu(A)$ for all $A\in\mathcal{O}(X);$
\item[{\rm (b)}] $\Theta$ is subadditive, i.e., $\Theta(A)\leq\Theta(B)+\Theta(C)$ for all $A,B,C\in\mathcal{O}(X)$ with $B,C\subset A$, $B\cap C=\emptyset$ and $\mu(A\setminus B\cup C)=0$.
\end{enumerate}
Then $\Theta$ is differentiable with respect to $\mu$, $d_\mu \Theta\in L^1_\mu(X)$ and
$$
\Theta^*(A)=\int_A d_\mu\Theta(x)d\mu(x)
$$
for all $A\in\mathcal{O}(X)$. 
\end{theorem}

As a direct consequence, we have

\begin{corollary}\label{CoroLLary-ViTali-EnVELopE-Theorem}
Let $\Theta:\mathcal{O}(X)\to\RR$ be a set function satisfying the assumptions {\rm (a)} and {\rm (b)} of Theorem {\rm\ref{Vitali-Envelope-Prop2}}. Then $\Theta$ and $\Theta^*$ are differentiable with respect to $\mu$ and $d_\mu\Theta^*=d_\mu\Theta$. 
\end{corollary}

\begin{proof}[\bf Proof of Theorem \ref{Vitali-Envelope-Prop2}]
First of all, From (a) we see that $-d_\mu\nu\leq d^-_\mu\Theta\leq d_\mu^+\Theta\leq d_\mu \nu$. Hence $d^-_\mu\Theta, d_\mu^+\Theta\in L^1_\mu(X)$ because $\nu$ is a finite Radon measure which is absolutely continuous with respect to the finite Radon measure $\mu$. So $\lambda^-(A),\lambda^+(A)\in\RR$ for all $A\in\mathcal{O}(X)$, where $\lambda^-,\lambda^+:\mathcal{O}(X)\to\RR$ are given by: 
\begin{trivlist}
\item $\displaystyle \lambda^-(A):=\int_A d^-_\mu\Theta(x)d\mu(x)$;
\item $\displaystyle \lambda^+(A):=\int_A d^+_\mu\Theta(x)d\mu(x)$.
\end{trivlist}
In what follows, we consider $\overline{\Theta}^*:\mathcal{O}(X)\to\RR$ defined by
\begin{equation}\label{Upper-Vitali-Envelope}
\overline{\Theta}^*(A):=\inf_{\eps>0}\sup\left\{\sum_{i\in I}\Theta(Q_i):\{Q_i\}_{i\in I}\in \mathcal{V}_\eps(A)\right\}.
\end{equation}
(It is clear that $\Theta^*\leq \overline{\Theta}^*$. In fact, we are going to prove that under the assumptions (a) and (b) of Theorem \ref{Vitali-Envelope-Prop2} we have $\Theta^*(A)=\overline\Theta^*(A)=\int_A d_\mu\Theta(x)d\mu(x)$ for all $A\in\mathcal{O}(X)$.) We divide the proof into three steps.

\medskip

\paragraph{\bf Step 1: proving that \boldmath$\Theta^*=\lambda^-$\unboldmath\ and \boldmath$\overline{\Theta}^*=\lambda^+$\unboldmath} Define $\theta^-,\theta^+:\mathcal{O}(X)\to\RR$ by:
\begin{trivlist}
\item $\theta^-(A):=\Theta(A)-\lambda^-(A)$;
\item $\theta^+(A):=\Theta(A)-\lambda^+(A)$.
\end{trivlist}
In what follows, $\theta^*$ (resp. $\overline{\theta}^*$) is defined by \eqref{Vitali-Envelope-DEF}  (resp. \eqref{Upper-Vitali-Envelope})  with $\Theta$ replaced by $\theta^-$ (resp. $\theta^+$).
\begin{lemma}\label{Lemma1-Vitali-Envelope-1}
Under the assumption {\rm (a)} of Theorem {\rm \ref{Vitali-Envelope-Prop2}} we have $\theta^*=\overline{\theta}^*=0$.
\end{lemma}
\begin{proof}[\bf Proof of Lemma \ref{Lemma1-Vitali-Envelope-1}]
We only prove that $\theta^*=0$. (The proof of $\overline{\theta}^*=0$ follows from similar arguments and is left to the reader.) 

First of all, from the assumption (a) it is clear that 
\begin{equation}\label{Lemma1-Vitali-Envelope-1-EqUaTiOn1}
|\theta^-(A)|\leq \hat \nu(A)
\end{equation}
for all $A\in\mathcal{O}(X)$, where $\hat\nu:=\nu+|\nu|$ is absolutely continuous with respect to $\mu$ (with $|\nu|$ denoting the total variation of $\nu$).

Secondly, we can assert that 
\begin{equation}\label{Lemma1-Vitali-Envelope-1-EqUaTiOn2}
d^-_\mu\theta^-=0, 
\end{equation}
where for any set function ${\rm s}:\mathcal{O}(X)\to\RR$, the function $d^-_\mu{\rm s}:X\to[-\infty,\infty[$ (resp. $d^+_\mu{\rm s}:X\to]-\infty,\infty]$) is defined by \eqref{Lower-Differential-Measure} (resp. \eqref{Upper-Differential-Measure}) with $\Theta$ replaced by ${\rm s}$. Indeed, for any $x\in X$, it is easily seen that
$$
d^-_\mu\Theta(x,\rho)-d^+_\mu\lambda^-(x,\rho)\leq d^-_\mu\theta^-(x,\rho)\leq d^-_\mu\Theta(x,\rho)-d^-_\mu\lambda^-(x,\rho).
$$
for all $\rho>0$, and letting $\rho\to 0$, we obtain
$$
d^-_\mu\Theta(x)-d^+_\mu\lambda^-(x)\leq d^-_\mu\theta^-(x)\leq d^-_\mu\Theta(x)-d^-_\mu\lambda^-(x).
$$
But $d^-_\mu\lambda^-(x)=d^+_\mu\lambda^-(x)=d^-_\mu\Theta(x)$, hence  $d^-_\mu\theta^-(x)=0$.

Finally, to conclude we prove that \eqref{Lemma1-Vitali-Envelope-1-EqUaTiOn1} and \eqref{Lemma1-Vitali-Envelope-1-EqUaTiOn2} imply $\theta^*=0$. For this, we are going to prove the following two assertions:
\begin{eqnarray}
&&\hbox{if $d^-_\mu\theta^-\leq 0$ then $\theta^*\leq 0$;}\label{Lemma1-Vitali-Envelope-1-EqUaTiOn3}\\
&&\hbox{under \eqref{Lemma1-Vitali-Envelope-1-EqUaTiOn1}, if $d^-_\mu\theta^-\geq 0$ then $\theta^*\geq 0$.}\label{Lemma1-Vitali-Envelope-1-EqUaTiOn4}
\end{eqnarray}
\paragraph{\em Proof of \eqref{Lemma1-Vitali-Envelope-1-EqUaTiOn3}} Fix $A\in\mathcal{O}(X)$. Fix any $\eps>0$. Then $d^-_\mu\theta^-<\eps$, and so in particular $\lim_{\rho\to 0}d^-_\mu\theta^-(x,\rho)<\eps$ for all $x\in A$. Hence, for each $x\in A$ there exists $\{\rho_{x,n}\}_n\subset]0,\eps[$ with $\rho_{x,n}\to0$ as $n\to\infty$ such that $d_\mu^-\theta^-(x,\rho_{x,n})<\eps$ for all $n\geq 1$. Taking Remark \ref{Remark2-ViTaLi-ENVELopE} into account, it follows that for each $x\in A$ and each $n\geq 1$ there is $Q_{x,n}\in {\rm Ba}(A,x,\rho_{x,n})$ such that for each $x\in A$ and each $n\geq 1$,
\begin{equation}\label{Lemma1-Vitali-Envelope-1-EqUaTiOn5}
{\theta^-(Q_{x,n})\over\mu(Q_{x,n})}<\eps.
\end{equation}
Moreover, since ${\rm diam}\big(\overline{Q}_{x,n}\big)={\rm diam}(Q_{x,n})\leq\rho_{x,n}$ for all $x\in A$ and all $n\geq 1$, we have  $\inf\big\{{\rm diam}\big(\overline{Q}_{x,n}\big):n\geq 1\big\}=0$ (where $\overline{Q}_{x,n}$ denotes the closed ball corresponding to the open ball $Q_{x,n}$). Let $\mathcal{F}_0$ be the family of closed balls of $X$ given by
$$
\mathcal{F}_0:=\Big\{\overline{Q}_{x,n}:x\in A\hbox{ and }n\geq 1\Big\}.
$$
As $X$ satisfies the Vitali covering theorem, from the above we deduce that there exists a disjointed countable subfamily $\{\overline{Q}_i\}_{i\in I_0}$ of closed balls of $\mathcal{F}_0$ (with $Q_i\subset A$, $\mu(\partial Q_i)=0$ and ${\rm diam}(Q_i)\in]0,\eps[$) such that $\mu\big(A\setminus\cup_{i\in I_0}\overline{Q}_i\big)=0$, which means that  $\{Q_i\}_{i\in I_0}\in\mathcal{V}_\eps(A)$. From \eqref{Lemma1-Vitali-Envelope-1-EqUaTiOn5} we see that $\theta^-(Q_i)<\eps\mu(Q_i)$ for all $i\in I_0$, hence 
$$
\sum_{i\in I_0}\theta^-(Q_i)\leq \eps\sum_{i\in I_0}\mu(Q_i)=\eps\mu(A).
$$
Consequently
$
\theta^{-,\eps}(A)\leq \eps\mu(A)
$
for all $\eps>0$, where $\theta^{-,\eps}$ is defined by \eqref{Pre-Vitali-Envelope-DEF} with $\Theta$ replaced by $\theta^-$, and letting $\eps\to 0$ we obtain $\theta^*(A)\leq 0$.

\smallskip

\paragraph{\em Proof of \eqref{Lemma1-Vitali-Envelope-1-EqUaTiOn4}} Fix $A\in\mathcal{O}(X)$. By Egorov's theorem, there exists a sequence $\{B_n\}_n$ of Borel subsets of $A$ such that:
\begin{eqnarray}
&& \lim_{n\to\infty}\mu(A\setminus B_n)=0;\label{Lemma1-Vitali-Envelope-1-EqUaTiOn6}\\
&& \lim_{\eps\to 0}\sup_{x\in B_n}\left|d_\mu^-\theta^-(x)-d_\mu^-\theta^-(x,\eps)\right|=0\hbox{ for all }n\geq 1.\label{Lemma1-Vitali-Envelope-1-EqUaTiOn7}
\end{eqnarray}
As $\hat\nu$ is absolutely continuous with respect to $\mu$, by \eqref{Lemma1-Vitali-Envelope-1-EqUaTiOn6} we have 
\begin{equation}\label{Lemma1-Vitali-Envelope-1-EqUaTiOn8}
\lim_{n\to\infty}\hat \nu(A\setminus B_n)=0.
\end{equation}
Moreover, as $d_\mu^-\theta^-\geq 0$, from \eqref{Lemma1-Vitali-Envelope-1-EqUaTiOn7} we deduce that
\begin{equation}\label{Lemma1-Vitali-Envelope-1-EqUaTiOn9}
\liminf_{\eps\to 0}\inf_{x\in B_n} d_\mu^-\theta^-(x,\eps)\geq 0\hbox{ for all }n\geq 1.
\end{equation}
Fix any $n\geq 1$ and any $\eps>0$. By definition of ${\rm \theta}^{-,\eps}$, there exists $\{Q_i\}_{i\in I}\in\mathcal{V}_\eps(A)$ such that
\begin{equation}\label{Lemma1-Vitali-Envelope-1-EqUaTiOn10}
\theta^{-,\eps}(A)>\sum_{i\in I}\theta^-(Q_i)-\eps.
\end{equation}
Set $I_n:=\big\{i\in I:Q_i\cap B_n\not=\emptyset\big\}$. Using \eqref{Lemma1-Vitali-Envelope-1-EqUaTiOn1} we have
\begin{eqnarray*}
\sum_{i\in I}\theta^-(Q_i)=\sum_{i\in I_n}\theta^-(Q_i)+\sum_{i\in I\setminus I_n}\theta^-(Q_i)&\geq& \sum_{i\in I_n}\theta^-(Q_i)-\sum_{i\in I\setminus I_n}\hat \nu(Q_i)\\
&\geq& \sum_{i\in I_n}{\theta^-(Q_i)\over\mu(Q_i)}\mu(Q_i)-\hat\nu\left(\cupp_{i\in I\setminus I_n}Q_i\right),
\end{eqnarray*}
and, choosing $x_i\in Q_i\cap B_n$ for each $i\in I_n$ and noticing that $\cupp_{i\in I\setminus I_n}Q_i\subset A\setminus B_n$,  it follows that 
\begin{eqnarray*}
 \sum_{i\in I}\theta^-(Q_i)&\geq& \sum_{i\in I_n}d_\mu^-\theta^-(x_i,\eps)\mu(Q_i)-\hat\nu(A\setminus B_n)\\
 &\geq&\inf_{x\in B_n} d_\mu^-\theta^-(x,\eps)\sum_{i\in I_n}\mu(Q_i)-\hat\nu(A\setminus B_n).
\end{eqnarray*}
Taking \eqref{Lemma1-Vitali-Envelope-1-EqUaTiOn10} into account, we conclude that
$$
\theta^{-,\eps}(A)\geq \inf_{x\in B_n} d_\mu^-\theta^-(x,\eps)\sum_{i\in I_n}\mu(Q_i)-\hat\nu(A\setminus B_n)-\eps
$$
for all $\eps>0$ and all $n\geq 1$, which gives $\theta^*(A)\geq 0$ by letting $\eps\to 0$ and using \eqref{Lemma1-Vitali-Envelope-1-EqUaTiOn9} and then by letting $n\to \infty$ and using \eqref{Lemma1-Vitali-Envelope-1-EqUaTiOn8}.
\end{proof}

\smallskip

As $\lambda^-$ and $\lambda^+$ are absolutely continuous with respect to $\mu$, it is easy to see that:
\begin{trivlist}
\item $\theta^*=\Theta^*-\lambda^-$;
\item $\overline{\theta}^*=\overline{\Theta}^*-\lambda^+$.
\end{trivlist}
Hence $\Theta^*=\lambda^-$ and $\overline{\Theta}^*=\lambda^+$ by Lemma \ref{Lemma1-Vitali-Envelope-1}.
\medskip

\paragraph{\bf Step 2: proving that \boldmath$\Theta^*=\overline{\Theta}^*$\unboldmath} We only need to prove that $\overline{\Theta}^*\leq \Theta^*$. For this, it is sufficient to show that for each open ball $Q$ of $X$ with $\mu(\partial Q)=0$, one has
\begin{equation}\label{Lemma1-Vitali-Envelope-1-EqUaTiOn11}
\Theta(Q)\leq\Theta^*(Q).
\end{equation}
Fix any $\eps >0$. By definition of $\Theta^\eps$, there exists $\{Q_i\}_{i\in I}\in \mathcal{V}_\eps(Q)$ such that
\begin{equation}\label{Lemma1-Vitali-Envelope-1-EqUaTiOn12}
\sum_{i\in I}\Theta(Q_i)\leq \Theta^\eps(Q)+\eps.
\end{equation}
Since $\mu\big(Q\setminus \cup_{i\in I}Q_i\big)=0$ there is a sequence $\{I_n\}_n$ of finite subsets of $I$ such that
\begin{equation}\label{Lemma1-Vitali-Envelope-1-EqUaTiOn13}
\lim_{n\to\infty}\mu\left(Q\setminus \cupp_{i\in I_n}Q_i\right)=\lim_{n\to\infty}\mu\left(\cupp_{i\in I\setminus I_n}Q_i\right)=0.
\end{equation}
Fix any $n\geq 1$. As $\Theta$ is subadditive, see the assumption (b), we have
$$
\Theta\left(\cupp_{i\in I_n}Q_i\right)\leq \sum_{i\in I_n}\Theta(Q_i).
$$
Moreover, $\mu\left(Q\setminus\big[(\cup_{i\in I_n}Q_i)\cup(Q\setminus\overline{\cup_{i\in I_n}Q_i})\big]\right)=0$ because $\mu(\partial Q_i)=0$ for all $i\in I_n$, hence
$$
\Theta(Q)\leq \Theta\left(\cupp_{i\in I_n}Q_i\right)+\Theta\left(Q\setminus\overline{\cupp_{i\in I_n}Q_i}\right)
$$
by using again the subadditivity of $\Theta$, and consequently
$$
\sum_{i\in I_n}\Theta(Q_i)\geq \Theta(Q)-\Theta\left(Q\setminus\overline{\cupp_{i\in I_n}Q_i}\right).
$$
Thus, using the assumption (a), we get
\begin{eqnarray*}
\sum_{i\in I}\Theta(Q_i)&=&\sum_{i\in I\setminus I_n}\Theta(Q_i)+\sum_{i\in I_n}\Theta(Q_i)\\
&\geq &\sum_{i\in I\setminus I_n}\Theta(Q_i)+\Theta(Q)-\Theta\left(Q\setminus\overline{\cupp_{i\in I_n}Q_i}\right)\\
&\geq&\Theta(Q)-\nu\left(\cupp_{i\in I\setminus I_n}Q_i\right)-\nu\left(Q\setminus\overline{\cupp_{i\in I_n}Q_i}\right).
\end{eqnarray*}
But, $\nu(\partial Q_i)=0$ for all $i\in I_n$ because $\nu$ is absolutely with respect to $\mu$, hence
$$
\nu\left(Q\setminus\overline{\cupp_{i\in I_n}Q_i}\right)=\nu\left(Q\setminus\cupp_{i\in I_n}Q_i\right)=\nu\left(\cupp_{i\in I\setminus I_n}Q_i\right),
$$
and so
\begin{equation}\label{Lemma1-Vitali-Envelope-1-EqUaTiOn14}
\sum_{i\in I}\Theta(Q_i)\geq \Theta(Q)-2\nu\left(\cupp_{i\in I\setminus I_n}Q_i\right).
\end{equation}
Combining \eqref{Lemma1-Vitali-Envelope-1-EqUaTiOn12} with \eqref{Lemma1-Vitali-Envelope-1-EqUaTiOn14} we conclude that
$$
\Theta(Q)\leq\Theta^\eps(Q)+2\nu\left(\cupp_{i\in I\setminus I_n}Q_i\right)+\eps,
$$
and \eqref{Lemma1-Vitali-Envelope-1-EqUaTiOn11} follows by letting $n\to\infty$ and using \eqref{Lemma1-Vitali-Envelope-1-EqUaTiOn13} and then by letting $\eps\to 0$.
\medskip

\paragraph{\bf Step 3: end of the proof} From steps 1 and 2 we have
$$
\int_X d_\mu^- \Theta(x)d\mu(x)=\Theta^*(X)=\overline{\Theta}^*(X)=\int_X d_\mu^+ \Theta(x)d\mu(x).
$$
Thus $\int_X(d_\mu^+\Theta(x)-d_\mu^-\Theta(x))d\mu(x)=0$. But $d_\mu^+\Theta\geq d_\mu^-\Theta$, i.e., $d_\mu^+\Theta- d_\mu^-\Theta\geq 0$, hence $d_\mu^+\Theta- d_\mu^-\Theta=0$, i.e., $d_\mu^+\Theta=d_\mu^-\Theta$, and the proof of Theorem \ref{Vitali-Envelope-Prop2} is complete.
\end{proof}


\subsection{Ru-usc integrands}

Let $(X,d,\mu)$ be a metric space measure, where $\mu$ is a positive Radon measure on $X$, and let $L:X\times\MM\to[0,\infty]$ be a Borel measurable integrand.  For each $x\in X$, we denote the effective domain of $L(x,\cdot)$ by $\LL_x$ and, for each $a\in L^1_{\rm loc,\mu}(X;]0,\infty])$, we define $\Delta_{L}^a:[0,1]\to]-\infty,\infty]$ by
 $$
 \Delta_{L}^a(t):=\sup_{x\in X}\sup_{\xi\in \LL_x}{L(x,t\xi)-L(x,\xi)\over a(x)+L(x,\xi)}.
 $$
 (When $(X,d)$ is compact, $L^1_{\rm loc,\mu}(X;]0,\infty])$ can be replaced by $L^1_{\mu}(X;]0,\infty])$.)
 \begin{definition}\label{ru-usc-Def}
 We say that $L$ is radially uniformly upper semicontinuous (ru-usc) if there exists $a\in L^1_{\rm loc,\mu}(X;]0,\infty])$ such that
 $$
 \limsup_{t\to 1^-}\Delta^a_L(t)\leq 0.
 $$

 \end{definition}
The concept of ru-usc integrand was introduced in \cite{oah10} and then developed in \cite{AHM11,oah-jpm-RM,oah-jpm12,jpm13,AHM14,AHMZ15}. (Recently, this concept was used by Duerinckx and Gloria in \cite{duerinckx-gloria16} to deal with stochastic homogenization of unbounded nonconvex integrals with convex growth.)
 
 \begin{remark}\label{ReMaRk-Delta-ru-usc-remARK1}
 If $L$ is ru-usc then $\limsup_{t\to1^-}L(x,t\xi)\leq L(x,\xi)$ for all $x\in X$ and all $\xi\in\LL_x$. On the other hand, if there exist $x\in X$ and $\xi\in\LL_x$ such that $L(x,\cdot)$ is lsc at $\xi$ then, for each $a\in L^1_{\rm loc}(U;]0,\infty])$, $\liminf_{t\to 1^-}\Delta^a_L(t)\geq 0$, and so if in addition $L$ is ru-usc then $\lim_{t\to 1^-}\Delta^a_L(t)=0$ for some $a\in L^1_{\rm loc,\mu}(X;]0,\infty])$.
 \end{remark}
 
 \begin{remark}
 If, for every $x\in X$, $L(x,\cdot)$ is convex and $0\in\LL_x$, then $L$ is ru-usc.
 \end{remark}

The interest of Definition \ref{ru-usc-Def} comes from the following theorem. (For a proof we refer to \cite[Theorem 3.5]{AHM11}.)

\begin{theorem}\label{Extension-Result-for-ru-usc-Functions}
If $L$ is ru-usc and if for every $x\in X$, 
\begin{equation}\label{Homothecie-Assumption-Bis}
t\overline{\LL}_x\subset{\rm int}(\LL_x)\hbox{ for all }t\in]0,1[
\end{equation}
 and $L(x,\cdot)$ is lsc on ${\rm int}(\LL_x)$, where ${\rm int}(\LL_x)$ denotes the interior of $\LL_x$, then{\rm:}
\begin{enumerate}[leftmargin=*]
\item[\rm(a)] $
\widehat{L}(x,\xi)=\left\{
\begin{array}{ll}
L(x,\xi)&\hbox{if }\xi\in{\rm int}(\LL_x)\\
\lim\limits_{t\to 1^-}L(x,t\xi)&\hbox{if }\xi\in\partial\LL_x\\
\infty&\hbox{otherwise{\rm;}}
\end{array}
\right.
$
\item[\rm(b)] $\widehat{L}$ is ru-usc{\rm;}
\item[\rm(c)] for every $x\in X$, $\widehat{L}(x,\cdot)$ is the lsc envelope of $L(x,\cdot)$.
\end{enumerate}
\end{theorem}

For each $\rho>0$, let $\mathcal{H}^\rho_\mu L:X\times\MM\to[0,\infty]$ be given by
$$
\mathcal{H}^\rho_\mu L(x,\xi):=\inf\left\{\mint_{Q_\rho(x)}L(y,\xi+\nabla_\mu w(y))d\mu(y):w\in W^{1,p}_{\mu,0}(Q_\rho(x);\RR^m)\right\}.
$$
Then, we have
$$
\mathcal{Q}_\mu L(x,\xi):=\limsup_{\rho\to 0}\mathcal{H}^\rho_\mu L(x,\xi)
$$
for all $x\in X$ and all $\xi\in\MM$. The following proposition shows that ru-usc is conserved under $\mu$-quasiconvexification.

\begin{proposition}\label{stability-of-ru-usc-by-mu-quasiconvexification}
If $L$ is ru-usc then $\mathcal{Q}_\mu L$ is ru-usc.
\end{proposition}

\begin{proof}[\bf Proof of Proposition \ref{stability-of-ru-usc-by-mu-quasiconvexification}]
Fix any $t\in[0,1]$, any $x\in X$ and any $\xi\in\mathcal{Q}_\mu\LL_x$ where $\mathcal{Q}_\mu\LL_x$ denotes the effective domain of $\mathcal{Q}_\mu L(x,\cdot)$. Then $\mathcal{Q}_\mu L(x,\xi)=\limsup_{\rho\to 0}\mathcal{H}^\rho_\mu L(x,\xi)<\infty$ and without loss of generality we can suppose that $\mathcal{H}^\rho_\mu L(x,\xi)<\infty$ for all $\rho>0$. 

Fix any $\rho>0$. By definition, there exists $\{w_n\}_n\subset W^{1,p}_{\mu,0}(Q_\rho(x);\RR^m)$ such that:
\begin{eqnarray}
&&\mathcal{H}^\rho_\mu L(x,\xi)=\lim_{n\to\infty}\mint_{Q_\rho(x)} L(y,\xi+\nabla_\mu w_n(y))d\mu(y);\label{Ru-usc-prop-mu-Quasi-Eq1}\\
&&\xi+\nabla_\mu w_n(y)\in \LL_y\hbox{ for all }n\geq 1\hbox{ and }\mu\hbox{-a.a. }y\in Q_{\rho}(x).\label{Ru-usc-prop-mu-Quasi-Eq2}
\end{eqnarray}
Moreover, for every $n\geq 1$,
$$
\mathcal{H}^\rho_\mu L(x,t\xi)\leq \mint_{Q_\rho(x)} L\big(y,t(\xi+\nabla_\mu w_n(y))\big)d\mu(y)
$$
since $t w_n\in W^{1,p}_{\mu,0}(Q_\rho(x);\RR^m)$, and so
\begin{equation}\label{Ru-usc-prop-mu-Quasi-Eq3}
\mathcal{H}^\rho_\mu L(x,t\xi)-\mathcal{H}^\rho_\mu L(x,\xi)\leq \liminf_{n\to\infty}\mint_{Q_\rho(x)}\big(L(y,t(\xi+\nabla_\mu w_n(y)))-L(y,\xi+\nabla_\mu w_n(y))\big)d\mu(y).
\end{equation}
But $L$ is ru-usc, i.e., there exists $a\in L^1_{\rm loc,\mu}(X;]0,\infty])$ such that $\limsup_{t\to 1^-}\Delta^a_L(t)\leq 0$ with $\Delta^a_L:[0,1]\to]-\infty,\infty]$ defined by $\Delta^a_L(t):=\sup_{z\in X}\sup_{\xi\in \LL_z}{L(z,t\xi)-L(z,\xi)\over a(z)+L(z,\xi)}$. So, taking \eqref{Ru-usc-prop-mu-Quasi-Eq2} into account, for every $n\geq 1$ and $\mu$-a.e. $y\in Q_\rho(x)$,
$$
L\big(y,t(\xi+\nabla_\mu w_n(y))\big)-L\big(y,\xi+\nabla_\mu w_n(y)\big)\leq\Delta_L^a(t)\big(a(y)+L(y,\xi+\nabla_\mu w_n^\rho(y))\big).
$$
Hence
\begin{eqnarray*}
\mint_{Q_\rho(x)}\big(L(y,t(\xi+\nabla_\mu w_n(y)))-L(y,\xi+\nabla_\mu w_n(y))\big)d\mu\hskip-2mm&\leq&\hskip-2mm\Delta_L^a(t)\left(\mint_{Q_\rho(x)}a(y)d\mu\right.\\
&&\hskip-2mm\left.+\mint_{Q_\rho(x)}L(y,\xi+\nabla_\mu w_n(y))d\mu\right)
\end{eqnarray*}
for all $n\geq 1$.  Letting $n\to\infty$ and using \eqref{Ru-usc-prop-mu-Quasi-Eq1} and \eqref{Ru-usc-prop-mu-Quasi-Eq3}, it follows that
\begin{equation}\label{Ru-usc-prop-mu-Quasi-Eq4}
\mathcal{H}^\rho_\mu L(x,t\xi)-\mathcal{H}^\rho_\mu L(x,\xi)\leq \Delta_L^a(t)\left(\mint_{Q_\rho(x)}a(y)d\mu(y)+\mathcal{H}^\rho_\mu L(x,\xi)\right)
\end{equation}
for all $\rho>0$, and so  
$$
\mathcal{Q}_\mu L(x,t\xi)-\mathcal{Q}_\mu L(x,\xi)\leq \Delta^a_L(t)\big({a}(x)+\mathcal{Q}_\mu L(x,\xi)\big)
$$
by letting $\rho\to 0$ in \eqref{Ru-usc-prop-mu-Quasi-Eq4}, which implies that $\Delta_{\mathcal{Q}_\mu L}^{{a}}(t)\leq \Delta^a_L(t)$ for all $t\in[0,1]$, and the proof is complete.
\end{proof}

\section{Proof of the integral representation theorem}

This section is devoted to the proof of Theorem \ref{MainTheorem} which is divided into six steps.

\medskip

\paragraph{\bf Step 1: integral representation of the \boldmath$\overline{E}$\unboldmath} 

For each $u\in W^{1,p}_\mu( X;\RR^m)$ we consider the set function $\mathcal{S}_u:\mathcal{O}( X)\to[0,\infty]$ given by
$$
 \mathcal{S}_u(A):=\overline{E}(u,A).
$$
Recall that $\mathfrak{G}$ is the effective domain of the functional $u\mapsto\int_X G(x,\nabla_\mu u(x))d\mu(x)$. 

\begin{lemma}
\label{Lemma-1-MT}
If \eqref{p-coercivity}, \eqref{growth-on-Gx}, \eqref{hyp-convexity}, \eqref{Gx-growth} and \eqref{ru-usc-condition} hold then
$$
\mathcal{S}_u(A)=\int_A\lambda_u(x)d\mu(x)
$$
for all $u\in \mathfrak{G}$ and all $A\in\mathcal{O}( X)$ with $\lambda_u\in L^1_\mu( X)$ given by
$$
\lambda_u(x)=\lim_{\rho\to 0}{\mathcal{S}_u(Q_\rho(x))\over\mu(Q_\rho(x))}.
$$
\end{lemma}

\begin{proof}[\bf Proof of Lemma \ref{Lemma-1-MT}]
Fix $u\in \mathfrak{G}$.  Using the right inequality in \eqref{Gx-growth} we see that
\begin{eqnarray}\label{ImPorTanT-EqUAtION}
\mathcal{S}_u(A)\leq \beta\mu(A)+\beta\int_A G(x,\nabla_\mu u(x))d\mu(x)<\infty
\end{eqnarray}
for all $A\in\mathcal{O}( X)$. Thus, the condition (d) of Lemma \ref{DeGiorgi-Letta-Lemma} is satisfied with $\nu=\beta\big(1+G(x,\nabla_\mu u(x))\big)d\mu$ 
(which is absolutely continuous with respect to $\mu$). On the other hand, it is easily seen that the conditions (a) and (b) of Lemma \ref{DeGiorgi-Letta-Lemma} are satisfied. Hence, the proof is completed by proving  the condition (c) of Lemma \ref{DeGiorgi-Letta-Lemma}, i.e., 
\begin{eqnarray}\label{Subadditivity-First-Goal}
\mathcal{S}_u(A\cup B)\leq \mathcal{S}_u(A)+\mathcal{S}_u(B)\hbox{ for all }A,B\in\mathcal{O}( X).
\end{eqnarray}
Indeed, by Lemma \ref{DeGiorgi-Letta-Lemma}, the set function $\mathcal{S}_u$ can be (uniquely) extended to a (finite) positive Radon measure which is absolutely continuous with respect to $\mu$, and the theorem follows by using Radon-Nikodym's theorem and then Lebesgue's differentiation theorem. 
\end{proof}

\begin{remark}\label{Remark-Lemma-1-MT}
In fact, Lemma \ref{Lemma-1-MT} establishes that for every $u\in\mathfrak{G}$, $E(u,\cdot)$ can be uniquely extended to a finite positive Radon measure on $ X$ which is absolutely continuous with respect to $\mu$.
\end{remark}

To show \eqref{Subadditivity-First-Goal} we need the following lemma.

\begin{lemma}\label{LeMMa-MaiN-TheOReM1}
If $U,V,Z,T\in\mathcal{O}( X)$ are such that $\overline{Z}\subset U$ and $T\subset V$,  then
\begin{eqnarray}\label{SubAddiTive-Goal}
\mathcal{S}_u(Z\cup T)\leq\mathcal{S}_u(U)+\mathcal{S}_u(V).
\end{eqnarray}
\end{lemma}

\begin{proof}[\bf Proof of Lemma \ref{LeMMa-MaiN-TheOReM1}]
Let $\{u_n\}_{n}$ and $\{v_n\}_{n}$ be two sequences in $W^{1,p}_\mu( X;\RR^m)$ such that:
\begin{eqnarray}
&& \|u_n-u\|_{L^p_\mu( X;\RR^m)}\to0;\label{PrOoF-MT2-EquA1}\\
&& \|v_n- u\|_{L^p_\mu( X;\RR^m)}\to0;\label{PrOoF-MT2-EquA2}\\
&& \lim_{n\to\infty}\int_UL(x,\nabla_\mu u_n(x))d\mu(x)=\mathcal{S}_u(U)<\infty;\label{PrOoF-MT2-EquA3}\\
&& \lim_{n\to\infty}\int_VL(x,\nabla_\mu v_n(x))d\mu(x)=\mathcal{S}_u(V)<\infty.\label{PrOoF-MT2-EquA4}
\end{eqnarray}
Since $L$ is $p$-coercive (see \eqref{p-coercivity} and the left inequality in \eqref{Gx-growth}), from \eqref{PrOoF-MT2-EquA3} and\eqref{PrOoF-MT2-EquA4} we see that $\sup_n\|\nabla_\mu u_n\|_{L^p_\mu( X;\MM)}<\infty$ and $\sup_n\|\nabla_\mu v_n\|_{L^p_\mu( X;\MM)}<\infty$. As $p>\DiM$, taking \eqref{PrOoF-MT2-EquA1} and \eqref{PrOoF-MT2-EquA2} into account, by Corollary \ref{L-infty-CSE} we can assert, up to a subsequence, that:
\begin{eqnarray}
&& \|u_n-u\|_{L^\infty_\mu( X;\RR^m)}\to0;\label{PrOoF-MT2-EquA1-bis}\\
&& \|v_n- u\|_{L^\infty_\mu( X;\RR^m)}\to0.\label{PrOoF-MT2-EquA2-bis}
\end{eqnarray}
Fix $\delta\in]0,{\rm dist}(Z,\partial U)[$ with $\partial U:=\overline{U}\setminus U$, fix any $q\geq 1$ and consider $W^-_i,W^+_i\subset  X$ given by:
\begin{trivlist} 
\item[]$W^-_i:=\left\{x\in  X:{\rm dist}(x,Z)\leq {\delta\over 3}+{(i-1)\delta\over 3q}\right\}$;
\item[]$W^+_i:=\left\{x\in  X:{\delta\over 3}+{i\delta\over 3q}\leq{\rm dist}(x,Z)\right\},$
\end{trivlist}
where $i\in\{1,\cdots,q\}$. For every $i\in\{1,\cdots,q\}$ there exists a Uryshon function $\varphi_i\in{\rm Lip}( X)$ for the pair $(W^+_i,W^-_i)$. Fix any $n\geq 1$ and define $w_n^i\in W^{1,p}_\mu( X;\RR^m)$ by 
\begin{equation}\label{Def-w-i-n}
w^i_n:=\varphi_iu_n+(1-\varphi_i)v_n.
\end{equation}
Fix any $t\in]0,1[$. Setting $W_i:= X\setminus (W^-_i\cup W^{+}_i)$ and using Theorem \ref{cheeger-theorem}(d) and \eqref{Mu-der-Prod} we have
$$
\nabla_\mu(tw_n^i)=t\nabla_\mu w_n^i=\left\{
\begin{array}{ll}
t\nabla_\mu u_n&\hbox{in }W^-_i\\
(1-t){t\over 1-t}D_\mu\varphi_i\otimes(u_n-v_n)+t\big(\varphi_i\nabla_\mu u_n+(1-\varphi_i)\nabla_\mu v_n\big)&\hbox{in }W_i\\
t\nabla_\mu v_n&\hbox{in }W^+_i.
\end{array}
\right.
$$
Noticing that $Z\cup T=((Z\cup T)\cap W^-_i)\cup(W\cap W_i)\cup(T\cap W^+_i)$ with $(Z\cup T)\cap W_i^-\subset U$, $T\cap W^+_i\subset V$ and $W:=T\cap\{x\in U:{\delta\over 3}<{\rm dist}(x,Z)<{2\delta\over 3}\}$ we deduce that for every $i\in\{1,\cdots,q\}$,
\begin{eqnarray}
\int_{Z\cup T}L(x,t\nabla_\mu w^i_n)d\mu&\leq&\int_UL(x,t\nabla_\mu u_n)d\mu+\int_VL(x,t\nabla_\mu v_n)d\mu\nonumber\label{LayERs-Eq1}\\
&&+\int_{W\cap W_i}L(x,t\nabla_\mu w^i_n)d\mu.\label{LayERs-Eq1}
\end{eqnarray}
Fix any $i\in\{1,\cdots,q\}$. From the right inequality in \eqref{Gx-growth} and the inequality \eqref{hyp-convexity} we see that 
\begin{eqnarray}
\int_{W\cap W_i}L(x,t\nabla_\mu w^i_n)d\mu&\leq&\beta\mu(W\cap W_i)+ \beta\int_{W\cap W_i}G(x,t\nabla_\mu w^i_n)d\mu\nonumber\\
&\leq& \beta(1+\gamma)\mu(W\cap W_i)\nonumber\\
&&+\beta\gamma\int_{W\cap W_i}G(x,\varphi_i\nabla_\mu u_n+(1-\varphi_i)\nabla_\mu v_n)d\mu\nonumber\\
&&+\beta\gamma\int_{W\cap W_i}G\left(x,{t\over 1-t}D_\mu\varphi_i\otimes(u_n-v_n)\right)d\mu,\nonumber
\end{eqnarray}
and by using again the inequality \eqref{hyp-convexity} and the left inequality in \eqref{Gx-growth} we obtain
\begin{eqnarray}
\int_{W\cap W_i}L(x,t\nabla_\mu w^i_n)d\mu&\leq&\beta(1+\gamma+\gamma^2)\mu(W\cap W_i)\nonumber\\
&&+{\beta\gamma^2\over \alpha}\left(\int_{W\cap W_i}L(x,\nabla_\mu u_n)d\mu+\int_{W\cap W_i}L(x,\nabla_\mu v_n)d\mu\right)\nonumber\\
&&+\beta\gamma\int_{W\cap W_i}G\left(x,{t\over 1-t}D_\mu\varphi_i\otimes(u_n-v_n)\right)d\mu.\label{LayERs-EquA2}
\end{eqnarray}
On the other hand, we have
$$
\left|{t\over 1-t}D_\mu\varphi_i(x)\otimes(u_n(x)-v_n(x))\right|\leq \left|{t\over 1-t}\right|\|D_\mu\varphi_i\|_{L^\infty_\mu( X)}\|u_n-v_n\|_{L^\infty_\mu( X;\RR^m)}
$$
for $\mu$-a.a. $x\in X$. But $\lim_{n\to\infty}\|u_n-v_n\|_{L^\infty_\mu( X;\RR^m)}=0$ by \eqref{PrOoF-MT2-EquA1-bis} and \eqref{PrOoF-MT2-EquA2-bis}, hence for each $t\in]0,1[$ and each $i\in\{1,\cdots,q\}$ there exists $n_{t,i}\geq 1$ such that
$$
\left|{t\over 1-t}D_\mu\varphi_i(x)\otimes(u_n(x)-v_n(x))\right|\leq r
$$
for $\mu$-a.a. $x\in X$ and all $n\geq n_{t,i}$ with $r>0$ given by \eqref{growth-on-Gx}. Hence
\begin{equation}\label{LayERs-EquA3}
\int_{W\cap W_i}G\left(x,{t\over 1-t}D_\mu\varphi_i\otimes(u_n-v_n)\right)d\mu\leq \int_{W\cap W_i}\sup_{|\xi|\leq r}G(x,\xi)d\mu(x)
\end{equation}
for all $n\geq N_{t,q}$ with $N_{t,q}=\max\{n_{t,i}:i\in\{1,\cdots,q\}\}$. Moreover, we have:
\begin{eqnarray}
&&\hskip-14mm\int_U L(x,t\nabla_\mu u_n)d\mu\leq \int_U L(x,\nabla_\mu u_n)d\mu+\Delta^a_L(t)\left(\int_U a(x)d\mu(x)+\int_U L(x,\nabla_\mu u_n)d\mu\right);\label{Equ-LayERs-EquA4}\\
&&\hskip-14mm\int_V L(x,t\nabla_\mu v_n)d\mu\leq  \int_V L(x,\nabla_\mu v_n)d\mu+\Delta^a_L(t)\left(\int_V a(x)d\mu(x)+\int_V L(x,\nabla_\mu v_n)d\mu\right),\label{Equ-LayERs-EquA5}
\end{eqnarray}
where $a\in L^1_\mu(X;]0,\infty])$ is given by \eqref{ru-usc-condition} (and $\limsup_{t\to 1^-}\Delta^a_L(t)\leq 0$ because $L$ is ru-usc).

Taking \eqref{LayERs-EquA3} into account and substituting \eqref{LayERs-EquA2}, \eqref{Equ-LayERs-EquA4} and \eqref{Equ-LayERs-EquA5} into \eqref{LayERs-Eq1} and then averaging these inequalities, it follows that for every $q\geq 1$, every $t\in]0,1[$ and every $n\geq N_{t,q}$, there exists $i_{n,t,q}\in\{1,\cdots,q\}$ such that
\begin{eqnarray}
\int_{Z\cup T}L(x,\nabla_\mu (tw_n^{i_{n,t,q}}))d\mu\hskip-2mm&\leq&\hskip-2mm\int_U L(x,\nabla_\mu u_n)d\mu+\Delta^a_L(t)\left(\int_U a(x)d\mu(x)+\int_U L(x,\nabla_\mu u_n)d\mu\right)\nonumber\\
&&\hskip-2mm+\int_V L(x,\nabla_\mu v_n)d\mu+\Delta^a_L(t)\left(\int_V a(x)d\mu(x)+\int_V L(x,\nabla_\mu v_n)d\mu\right)\nonumber\\
&&\hskip-2mm+{c\over q}\left(\int_X\sup_{|\xi|\leq r}G(x,\xi)d\mu+\int_U L(x,\nabla_\mu u_n)d\mu+\int_VL(x,\nabla_\mu v_n)d\mu\right)\nonumber
\end{eqnarray}
with $c=\max\big\{\beta(1+\gamma+\gamma^2)+1, {\beta\gamma^2\over\alpha}\big\}$, where $\int_X\sup_{|\xi|\leq r}G(x,\xi)d\mu<\infty$ by \eqref{growth-on-Gx}. Thus, letting $n\to\infty$, $t\to 1^-$ and $q\to\infty$ and using \eqref{PrOoF-MT2-EquA3} and \eqref{PrOoF-MT2-EquA4}, we get
\begin{equation}\label{FinAL-inEquAtIOn-SteP1}
\limsup_{q\to\infty}\limsup_{t\to 1^-}\limsup_{n\to\infty}\int_{Z\cup T}L(x,\nabla_\mu (tw_n^{i_{n,t,q}}))d\mu\leq \mathcal{S}_u(U)+\mathcal{S}_u(V).
\end{equation}
On the other hand, taking \eqref{Def-w-i-n} into account and using \eqref{PrOoF-MT2-EquA1} and \eqref{PrOoF-MT2-EquA2} we see that
$$
\lim_{q\to\infty}\lim_{t\to 1^-}\lim_{n\to\infty}\|tw_n^{i_{n,t,q}}-u\|_{L^p_\mu( X;\RR^m)}=0.
$$
By diagonalization, there exist increasing mappings $n\mapsto t_n$ and $n\mapsto q_n$ with $t_n\to 1^-$ and $q_n\to\infty$ such that:
\begin{eqnarray*}
&&\hskip-5.5mm\liminf_{n\to\infty}\int_{Z\cup T}L(x,\nabla_\mu \hat w_n)d\mu\leq \limsup_{n\to\infty}\int_{Z\cup T}L(x,\nabla_\mu \hat w_n)d\mu\leq \limsup_{q\to\infty}\limsup_{t\to 1^-}\limsup_{n\to\infty}\int_{Z\cup T}L(x,\nabla_\mu (tw_n^{i_{n,t,q}}))d\mu;\\
&&\hskip-5.5mm\lim_{n\to\infty}\|\hat w_n-u\|_{L^p_\mu( X;\RR^m)}=0,
\end{eqnarray*}
where $\hat w_n:=t_nw_n^{i_{n,t_n,q_n}}$. Hence
$$
\mathcal{S}_u(Z\cup T)\leq \limsup_{q\to\infty}\limsup_{t\to 1^-}\limsup_{n\to\infty}\int_{Z\cup T}L(x,\nabla_\mu (tw_n^{i_{n,t,q}}))d\mu,
$$
and \eqref{SubAddiTive-Goal} follows from \eqref{FinAL-inEquAtIOn-SteP1}. 
\end{proof}

\medskip

We now prove \eqref{Subadditivity-First-Goal}. Fix $A,B\in\mathcal{O}(X)$. Fix any $\eps>0$ and consider $C,D\in\mathcal{O}(X)$ such that $\overline{C}\subset A$, $\overline{D}\subset B$ and 
$$
\beta\mu(E)+\beta\int_E g_u(x)d\mu(x)<\eps
$$ 
with $E:=A\cup B\setminus\overline{C\cup D}$. Then $\mathcal{S}_u(E)\leq\eps$ by \eqref{ImPorTanT-EqUAtION}. Let $\hat C,\hat D\in\mathcal{O}(X)$ be such that $\overline{C}\subset\hat C$, $\overline{\hat C}\subset A$, $\overline{D}\subset\hat D$ and $\overline{\hat D}\subset B$. Applying Lemma \ref{LeMMa-MaiN-TheOReM1} with $U=\hat C\cup\hat D$, $V=T=E$ and $Z=C\cup D$ (resp. $U=A$, $V=B$, $Z=\hat C$ and $T=\hat D$) we obtain
$$
\mathcal{S}_u(A\cup B)\leq\mathcal{S}_u(\hat C\cup\hat D)+\eps\hbox{ \big(resp. }\mathcal{S}_u(\hat C\cup\hat D)\leq\mathcal{S}_u(A)+\mathcal{S}_u(B)\big),
$$
and \eqref{Subadditivity-First-Goal} follows by letting $\eps\to0$. $\blacksquare$

\medskip

\paragraph{\bf Step 2: another formula for \boldmath$\overline{E}$\unboldmath} Let $\overline{E}_0:W^{1,p}_\mu(X;\RR^m)\times\mathcal{O}(X)\to[0,\infty]$ given by
$$
\overline{E}_0(u,A):=\inf\left\{\liminf_{n\to\infty}E(u_n,A):W^{1,p}_{\mu,0}(A;\RR^m)\ni u_n-u\stackrel{L^p_\mu}{\to}0\right\}.
$$
\begin{lemma}\label{Lemma-2-MT}
If \eqref{p-coercivity}, \eqref{growth-on-Gx}, \eqref{hyp-convexity}, \eqref{Gx-growth} and \eqref{ru-usc-condition} hold then for every $u\in \mathfrak{G}$ and every $A\in\mathcal{O}(X)$, one has{\rm:}
\begin{eqnarray}
&&\overline{E}(u,A)\leq\overline{E}_0(u,A);\label{GOal-Lemma-Bis-1}\\
&&\overline{E}_0(tu,A)\leq \big(1+\Delta^a_L(t)\big)\overline{E}(u,A)+\Delta_L^a(t)\|a\|_{L^1_\mu(A)}\hbox{ for all }t\in]0,1[,\label{GOal-Lemma-Bis}
\end{eqnarray}
where $a\in L^1_\mu(X;]0,\infty])$ is given by \eqref{ru-usc-condition}. As a direct consequence we have
\begin{eqnarray}\label{GOal-Lemma-Bis-FiNaL}
\overline{E}(u,A)=\lim_{t\to1^-}\overline{E}_0(tu,A)
\end{eqnarray}
for all $u\in \mathfrak{G}$ and all $A\in\mathcal{O}(X)$.
\end{lemma}

\begin{proof}[\bf Proof of Lemma \ref{Lemma-2-MT}]
Fix $u\in \mathfrak{G}$ and $A\in\mathcal{O}(X)$. Clearly, $\overline{E}_0(u;A)\geq \overline{E}(u,A)$ because $W^{1,p}_{\mu,0}(A;\RR^m)\subset W^{1,p}_\mu(X;\RR^m)$. So \eqref{GOal-Lemma-Bis-1} is satisfied. Thus, it remains to prove \eqref{GOal-Lemma-Bis}. Let $\{u_n\}_n\subset W^{1,p}_\mu(X;\RR^m)$ be such that:
\begin{eqnarray}
&& u_n\to u\hbox{ in }L^p_\mu(X;\RR^m)\label{PrOoF-MT2-EquA1-BBiSS};\\
&& \lim_{n\to\infty}\int_A L(x,\nabla_\mu u_n(x))d\mu(x)=\overline{E}(u,A)<\infty.\label{PrOoF-MT2-EquA3-BiS}
\end{eqnarray}

\begin{remark}
Without loss of generality we can always suppose that $\overline{E}(u,A)<\infty$. So, in fact, if we futhermore assume that \eqref{lsc-intGx} is satisfied, Lemma \ref{Lemma-2-MT} remains true if we replace ``$u\in\mathfrak{G}$" by ``$u\in W^{1,p}_\mu(X;\RR^m)$".
\end{remark}

Since $L$ is $p$-coercive (see \eqref{p-coercivity} and the left inequality in \eqref{Gx-growth}), from \eqref{PrOoF-MT2-EquA3-BiS} we see that $\sup_n\|\nabla_\mu u_n\|_{L^p_\mu( X;\MM)}<\infty$. As $p>\DiM$, taking \eqref{PrOoF-MT2-EquA1-BBiSS} into account, by Corollary \ref{L-infty-CSE} we can assert, up to a subsequence, that:
\begin{eqnarray}
&& \|u_n-u\|_{L^\infty_\mu( X;\RR^m)}\to0.\label{PrOoF-MT2-EquA1-BiS}
\end{eqnarray}
Fix $\delta>0$ and set $A_\delta:=\{x\in A:{\rm dist}(x,\partial A)>\delta\}$ with $\partial A:=\overline{A}\setminus A$. Fix any $n\geq 1$ and any $q\geq 1$ and consider $W^-_i,W^+_i\subset X$ given by
\begin{trivlist} 
\item[]$W^-_i:=\left\{x\in X:{\rm dist}(x,A_\delta)\leq {\delta\over 3}+{(i-1)\delta\over 3q}\right\}$;
\item[]$W^+_i:=\left\{x\in X:{\delta\over 3}+{i\delta\over 3q}\leq{\rm dist}(x,A_\delta)\right\}$,
\end{trivlist}
where $i\in\{1,\cdots,q\}$. (Note that $W^-_i\subset A$.) For every $i\in\{1,\cdots,q\}$ there exists a Uryshon function $\varphi_i\in {\rm Lip}(X)$ for the pair $(W^+_i,W^-_i)$. Define $w_n^i:X\to\RR^m$  by 
\begin{equation}\label{Def-w-i-n-BiS}
w^i_n:=\varphi_iu_n+(1-\varphi_i)u.
\end{equation}
Then $w_n^i-u\in W^{1,p}_{\mu,0}(A;\RR^m)$. Fix any $t\in]0,1[$. Setting $W_i:=X\setminus (W^-_i\cup W^{+}_i)\subset A$ and using Theorem \ref{cheeger-theorem}(d) and \eqref{Mu-der-Prod} we have
$$
\nabla_\mu(tw^i_n)=t\nabla_\mu w_n^i=\left\{
\begin{array}{ll}
t\nabla_\mu u_n&\hbox{in }W^-_i\\
(1-t){t\over 1-t}D_\mu\varphi_i\otimes(u_n-u)+t\big(\varphi_i\nabla_\mu u_n+(1-\varphi_i)\nabla_\mu u\big)&\hbox{in }W_i\\
t\nabla_\mu u&\hbox{in }W^+_i.
\end{array}
\right.
$$
Noticing that $A=W^-_i\cup W_i\cup(A\cap W^+_i)$ we deduce that for every $i\in\{1,\cdots,q\}$,
\begin{eqnarray}
\int_A L(x,t\nabla_\mu w^i_n)d\mu&\leq&\int_A L(x,t\nabla_\mu u_n)d\mu+\int_{A\cap W^+_i}L(x,t\nabla_\mu u)d\mu\label{LayERs-Eq1-BiS}\\
&&+\int_{W_i} L(x,t\nabla_\mu w^i_n)d\mu.\nonumber
\end{eqnarray}
Fix any $q\in\{1,\cdots,q\}$. From the right inequality in \eqref{Gx-growth} and the inequality \eqref{hyp-convexity} we see that 
\begin{eqnarray}
\int_{W_i}L(x,t\nabla_\mu w^i_n)d\mu&\leq&\beta\mu(W_i)+ \beta\int_{W_i}G(x,t\nabla_\mu w^i_n)d\mu\nonumber\\
&\leq& \beta(1+\gamma)\mu(W_i)\nonumber\\
&&+\beta\gamma\int_{W_i}G(x,\varphi_i\nabla_\mu u_n+(1-\varphi_i)\nabla_\mu u)d\mu\nonumber\\
&&+\beta\gamma\int_{W_i}G\left(x,{t\over 1-t}D_\mu\varphi_i\otimes(u_n-u)\right)d\mu,\nonumber
\end{eqnarray}
and by using again the inequality \eqref{hyp-convexity} and the left inequality in \eqref{Gx-growth} we obtain
\begin{eqnarray}
\int_{W_i}L(x,t\nabla_\mu w^i_n)d\mu&\leq&\beta(1+\gamma+\gamma^2)\mu(W_i)\nonumber\\
&&+{\beta\gamma^2\over \alpha}\left(\int_{W_i}L(x,\nabla_\mu u_n)d\mu+\int_{W_i}L(x,\nabla_\mu u)d\mu\right)\nonumber\\
&&+\beta\gamma\int_{W_i}G\left(x,{t\over 1-t}D_\mu\varphi_i\otimes(u_n-u)\right)d\mu.\label{LayERs-EquA2-BBiiSS}
\end{eqnarray}

\begin{remark}
As $u\in\mathfrak{G}$ and \eqref{Gx-growth} holds, we have $\int_E L(x,\nabla_\mu u)d\mu<\infty$ for all $E\in\mathcal{O}(X)$.
\end{remark}

On the other hand, we have
$$
\left|{t\over 1-t}D_\mu\varphi_i(x)\otimes(u_n(x)-u(x))\right|\leq \left|{t\over 1-t}\right|\|D_\mu\varphi_i\|_{L^\infty_\mu( X)}\|u_n-u\|_{L^\infty_\mu( X;\RR^m)}
$$
for $\mu$-a.a. $x\in X$. But $\lim_{n\to\infty}\|u_n-u\|_{L^\infty_\mu( X;\RR^m)}=0$ by \eqref{PrOoF-MT2-EquA1-BiS}, hence for each $i\in\{1,\cdots,q\}$ there exists $n_{i}\geq 1$ such that
$$
\left|{t\over 1-t}D_\mu\varphi_i(x)\otimes(u_n(x)-u(x))\right|\leq r
$$
for $\mu$-a.a. $x\in X$ and all $n\geq n_{i}$ with $r>0$ given by \eqref{growth-on-Gx}. Hence
\begin{equation}\label{LayERs-EquA3-BBiiSS}
\int_{W_i}G\left(x,{t\over 1-t}D_\mu\varphi_i\otimes(u_n-u)\right)d\mu\leq \int_{W_i}\sup_{|\xi|\leq r}G(x,\xi)d\mu(x)
\end{equation}
for all $n\geq N_{q}$ with $N_{q}=\max\{n_{i}:i\in\{1,\cdots,q\}\}$. Moreover, we have:
\begin{eqnarray}
&&\hskip-14mm\int_A L(x,t\nabla_\mu u_n)d\mu\leq \int_A L(x,\nabla_\mu u_n)d\mu+\Delta^a_L(t)\left(\int_A a(x)d\mu(x)+\int_A L(x,\nabla_\mu u_n)d\mu\right);\label{Equ-LayERs-EquA4-BBiiSS}\\
&&\hskip-14mm\int_{A\cap W^+_i} L(x,t\nabla_\mu u)d\mu\leq  \int_{A\cap W^+_i} L(x,\nabla_\mu u)d\mu\nonumber\\
&&\hskip27mm+\Delta^a_L(t)\left(\int_{A\cap W^+_i} a(x)d\mu(x)+\int_{A\cap W^+_i} L(x,\nabla_\mu u)d\mu\right),\label{Equ-LayERs-EquA5-BBiiSS}
\end{eqnarray}
where $a\in L^1_\mu(X;]0,\infty])$ is given by \eqref{ru-usc-condition} (and $\limsup_{t\to 1^-}\Delta^a_L(t)\leq 0$ because $L$ is ru-usc).

Taking \eqref{LayERs-EquA3-BBiiSS} into account and substituting \eqref{LayERs-EquA2-BBiiSS}, \eqref{Equ-LayERs-EquA4-BBiiSS} and \eqref{Equ-LayERs-EquA5-BBiiSS} into \eqref{LayERs-Eq1-BiS} and then averaging these inequalities, it follows that for every $q\geq 1$ and every $n\geq N_{q}$, there exists $i_{n,q}\in\{1,\cdots,q\}$ such that
\begin{eqnarray}
\int_{A}L(x,\nabla_\mu (tw_n^{i_{n,q}}))d\mu\hskip-1.5mm&\leq&\hskip-1.5mm\int_A L(x,\nabla_\mu u_n)d\mu+\Delta^a_L(t)\left(\int_A a(x)d\mu(x)+\int_A L(x,\nabla_\mu u_n)d\mu\right)\nonumber\\
&&+{1\over q}\left[\int_{A} L(x,\nabla_\mu u)d\mu+\Delta^a_L(t)\left(\int_{A} a(x)d\mu(x)+\int_{A} L(x,\nabla_\mu u)d\mu\right)\right]\nonumber\\
&&+{c\over q}\left(\int_A\sup_{|\xi|\leq r}G(x,\xi)d\mu+\int_A L(x,\nabla_\mu u_n)d\mu+\int_AL(x,\nabla_\mu u)d\mu\right)\nonumber
\end{eqnarray}
with $c=\max\big\{\beta(1+\gamma+\gamma^2)+1, {\beta\gamma^2\over\alpha}\big\}$, where $\int_A\sup_{|\xi|\leq r}G(x,\xi)d\mu<\infty$ by \eqref{growth-on-Gx}. Thus, letting $n\to\infty$ and $q\to\infty$ and using \eqref{PrOoF-MT2-EquA3-BiS}, we get
\begin{equation}\label{FinAL-inEquAtIOn-SteP1-BBiiSS}
\limsup_{q\to\infty}\limsup_{n\to\infty}\int_{A}L(x,\nabla_\mu (tw_n^{i_{n,q}}))d\mu\leq \big(1+\Delta^a_L(t)\big)\overline{E}(u,A) + \Delta^a_L(t)\int_A a(x)d\mu(x).
\end{equation}
On the other hand, taking \eqref{Def-w-i-n-BiS} into account and using \eqref{PrOoF-MT2-EquA1-BBiSS} we see that
$$
\lim_{q\to\infty}\lim_{n\to\infty}\|tw_n^{i_{n,q}}-tu\|_{L^p_\mu( X;\RR^m)}=0.
$$
By diagonalization, there exists an increasing mapping $n\mapsto q_n$ with $q_n\to\infty$ such that:
\begin{eqnarray*}
&&\hskip-5.5mm\liminf_{n\to\infty}\int_{A}L(x,\nabla_\mu \hat w_n)d\mu\leq \limsup_{n\to\infty}\int_{A}L(x,\nabla_\mu \hat w_n)d\mu\leq \limsup_{q\to\infty}\limsup_{n\to\infty}\int_{A}L(x,\nabla_\mu (tw_n^{i_{n,q}}))d\mu;\\
&&\hskip-5.5mm\lim_{n\to\infty}\|\hat w_n-u\|_{L^p_\mu( X;\RR^m)}=0,
\end{eqnarray*}
where $\hat w_n:=tw_n^{i_{n,q_n}}$ is such that  $\hat w_n-tu\in W^{1,p}_{\mu,0}(A;\RR^m)$. Hence 
$$
\overline{E}_0(tu,A)\leq \limsup_{q\to\infty}\limsup_{n\to\infty}\int_{A}L(x,\nabla_\mu (tw_n^{i_{n,q}}))d\mu,
$$
and \eqref{GOal-Lemma-Bis} follows from \eqref{FinAL-inEquAtIOn-SteP1-BBiiSS}. 

From \eqref{GOal-Lemma-Bis} we deduce that 
$$
\limsup_{t\to 1^-}\overline{E}_0(tu,A)\leq \overline{E}(u,A).
$$
Moreover, from \eqref{GOal-Lemma-Bis-1} we have
$$
\overline{E}(u,A)\leq\liminf_{t\to 1^-}\overline{E}(tu,A)\leq \liminf_{t\to 1^-}\overline{E}_0(tu,A),
$$
which gives \eqref{GOal-Lemma-Bis-FiNaL}.
\end{proof}

\medskip

\paragraph{\bf Step 3: using the Vitali envelope} For each $u\in W^{1,p}_\mu(X;\RR^m)$ we consider the set function $\widecheck{{\rm m}}_u:\mathcal{O}(X)\to[0,\infty]$ defined by
\begin{equation}\label{DeFINItIon-Of-widecheck-m-u}
\widecheck{{\rm m}}_u(A):=\limsup_{t\to 1^-}{\rm m}_{tu}(A),
\end{equation}
where, for each $z\in W^{1,p}_\mu(X;\RR^m)$, ${\rm m}_z:\mathcal{O}(X)\to[0,\infty]$ is given by 
\begin{equation}\label{DeFINItIon-Of-widecheck-m-u-bis}
{\rm m}_z(A):=\inf\Big\{E(v,A):v-z\in W^{1,p}_{\mu,0}(A;\RR^m)\Big\}.
\end{equation}
For each $\eps>0$ and each $A\in\mathcal{O}(X)$, denote the class of countable families $\{Q_i:=Q_{\rho_i}(x_i)\}_{i\in I}$ of disjoint open balls of $A$ with $x_i\in A$, $\rho_i={\rm diam}(Q_i)\in]0,\eps[$ and $\mu(\partial Q_i)=0$ such that $\mu(A\setminus\cup_{i\in I}Q_i)=0$ by $\mathcal{V}_\eps(A)$, consider $\widecheck{{\rm m}}_u^\eps:\mathcal{O}(X)\to[0,\infty]$ given by
$$
\widecheck{\rm m}_u^\eps(A):=\inf\left\{\sum_{i\in I}\widecheck{\rm m}_u(Q_i):\{Q_i\}_{i\in I}\in \mathcal{V}_\eps(A)\right\},
$$
and define $\widecheck{{\rm m}}^*_u:\mathcal{O}(X)\to[0,\infty]$ by
$$
\widecheck{{\rm m}}^*_u(A):=\sup_{\eps>0}\widecheck{{\rm m}}^\eps_u(A)=\lim_{\eps\to0}\widecheck{{\rm m}}_u^\eps(A).
$$
The set function $\widecheck{{\rm m}}^*_u$ is called the Vitali envelope of $\widecheck{{\rm m}}_u$, see \S3.3  for more details. (Note that as $X$ satisfies the Vitali covering theorem, see Proposition \ref{Fundamental-Proposition-for-CalcVar-in-MMS}(c) and Remark \ref{ReMArK-VItALi-For-OpEN-SEtS}, we have $\mathcal{V}_\eps(A)\not=\emptyset$ for all $A\in\mathcal{O}(X)$ and all $\eps>0$.) 

\begin{lemma}\label{Lemma-3-MT}
If \eqref{p-coercivity}, \eqref{growth-on-Gx}, \eqref{hyp-convexity}, \eqref{Gx-growth} and \eqref{ru-usc-condition} hold then 
\begin{equation}\label{Eq-Lemma-3-MT-bis}
\overline{E}(u,A)=\widecheck{{\rm m}}^*_{u}(A)
\end{equation}
for all $u\in \mathfrak{G}$ and all $A\in\mathcal{O}(X)$.
\end{lemma}
\begin{proof}[\bf Proof of Lemma \ref{Lemma-3-MT}]
Fix $u\in \mathfrak{G}$. Given any $A\in\mathcal{O}(X)$, it is easy to see that ${\rm m}_{tu}(A)\leq \overline{E}_0(tu,A)$ for all $t\in]0,1[$, hence 
$$
\widecheck{{\rm m}}_u(A)=\limsup_{t\to 1^-}{\rm m}_{tu}(A)\leq \lim_{t\to 1^-}\overline{E}_0(tu,A)=\overline{E}(u,A)
$$
by Lemma \ref{Lemma-2-MT}, and consequently
$$
\widecheck{{\rm m}}^*_u(A)\leq \overline{E}(u,A)
$$
because in the proof of Lemma \ref{Lemma-1-MT} it is established that $\overline{E}(u,\cdot)$ can be uniquely extended to a finite positive Radon measure on $X$, see Remark \ref{Remark-Lemma-1-MT}. Hence, to establish \eqref{Eq-Lemma-3-MT-bis}, it remains to prove that
\begin{equation}\label{EEEEqqqq}
\overline{E}(u,A)\leq \widecheck{\rm {m}}^*_u(A)
\end{equation}

with $\widecheck{\rm{m}}^*_u(A)<\infty$. Fix any $\eps>0$. By definition of $\widecheck{{\rm m}}^\eps_u(A)$ there exists $\{Q_{i}\}_{i\in I}\in\mathcal{V}_\eps(A)$  such that 
\begin{equation}\label{EEEqqq1}
\sum_{i\in I}\widecheck{\rm m}_u(Q_{i})\leq \widecheck{{\rm m}}^\eps_u(A)+{\eps\over 2}.
\end{equation}
Fix any $t>0$. For each $i\in I$, by definition of ${\rm{m}}_{tu}(Q_i)$ there exists $v_{t}^i\in W^{1,p}_{\mu}(Q_{i};\RR^m)$ such that $v_{t}^i-tu\in W^{1,p}_{\mu,0}(Q_{i};\RR^m)$ and
\begin{equation}\label{EEEqqq1-bis}
E(v_{t}^i,Q_i)\leq {\rm {m}}_{tu}(Q_i)+{\eps \mu(Q_i)\over 2\mu(A)}.
\end{equation}
Define $u^\eps_t:X\to\RR^m$ by 
$$
u^\eps_t:=\left\{
\begin{array}{ll}
tu&\hbox{in }X\setminus A\\
v_t^i&\hbox{in }Q_{i}.
\end{array}
\right.
$$
Then $u^\eps_t-tu\in W^{1,p}_{\mu,0}(A;\RR^m)$. Moreover, because of Proposition \ref{Fundamental-Proposition-for-CalcVar-in-MMS}(a), $\nabla_\mu u^\eps_t(x)=\nabla_\mu v^i_t(x)$ for $\mu$-a.e. $x\in Q_i$. From \eqref{EEEqqq1-bis} we see that
$$
E(u_{t}^\eps,A)\leq \sum_{i\in I}{\rm {m}}_{tu}(Q_i)+{\eps\over 2},
$$
hence
$
\limsup_{t\to 1^-}E(u_{t}^\eps,A)\leq \widecheck{{\rm m}}^\eps_u(A)+{\eps}
$
by using \eqref{EEEqqq1}, and consequently
\begin{equation}\label{Diag-SteP3-1}
\limsup_{\eps\to 0}\limsup_{t\to 1^-}E(u_{t}^\eps,A)\leq \widecheck{{\rm m}}^*_u(A).
\end{equation}
On the other hand, we have
\begin{eqnarray*}
\|u_t^\eps-u\|^{p}_{L^{\chi p}_\mu(X;\RR^m)}&\leq& 2^p\left[\|u_t^\eps-tu\|^{p}_{L^{\chi p}_\mu(X;\RR^m)}+\|tu-u\|^{p}_{L^{\chi p}_\mu(X;\RR^m)}\right]\\
&=&2^p\left[\left(\int_A|u_t^\eps-tu|^{\chi p}d\mu\right)^{1\over\chi}+(1-t)^p\|u\|^{p}_{L^{\chi p}_\mu(X;\RR^m)}\right]\\
&=&2^p\left[\left(\sum_{i\in I}\int_{Q_{i}}|v_t^i-tu|^{\chi p}d\mu\right)^{1\over\chi}+(1-t)^p\|u\|^{p}_{L^{\chi p}_\mu(X;\RR^m)}\right]\\
&\leq&2^p\left[\sum_{i\in I}\left(\int_{Q_{i}}|v_t^i-tu|^{\chi p}d\mu\right)^{1\over\chi}+(1-t)^p\|u\|^{p}_{L^{\chi p}_\mu(X;\RR^m)}\right]
\end{eqnarray*}
with $\chi\geq 1$ given by \eqref{Poincare-Inequality}. As $X$ supports a $p$-Sobolev inequality, see Proposition \ref{Fundamental-Proposition-for-CalcVar-in-MMS}(b), and ${\rm diam}(Q_i)\in]0,\eps[$ for all $i\in I$, we have
$$
\sum_{i\in I}\left(\int_{Q_{i}}|v_t^i-tu|^{\chi p}d\mu\right)^{1\over\chi}\leq\eps^{p} C_S^{p}\sum_{i\in I}\int_{Q_{i}}|\nabla_\mu v_t^i-t\nabla_\mu u|^pd\mu
$$
with $C_S>0$ given by \eqref{Poincare-Inequality}, hence
$$
\sum_{i\in I}\left(\int_{Q_{i}}|v_t^i-tu|^{\chi p}d\mu\right)^{1\over\chi}\leq 2^p\eps^{p} C_S^p\left(\sum_{i\in I}\int_{Q_{i}}|\nabla_\mu v_t^i|^pd\mu+t^p\int_A|\nabla_\mu u|^pd\mu\right),
$$
and consequently
\begin{eqnarray}\label{EEEqqq2}
\|u_t^\eps-u\|^{p}_{L^{\chi p}_\mu(X;\RR^m)}&\leq &2^{2p}\eps^{p} C_S^p\left(\sum_{i\in I}\int_{Q_{i}}|\nabla_\mu v_t^i|^pd\mu+t^p\int_A|\nabla_\mu u|^pd\mu\right)\nonumber\\
&&+2^p(1-t)^p\|u\|^{p}_{L^{\chi p}_\mu(X;\RR^m)}.\label{EEEqqq2}
\end{eqnarray}
Taking \eqref{p-coercivity}, the left inequality in \eqref{Gx-growth}, \eqref{EEEqqq1} and \eqref{EEEqqq1-bis} into account, from \eqref{EEEqqq2} we deduce that
$$
\limsup_{t\to 1^-}\|u_t^\eps-u\|^p_{L^{\chi p}_\mu(X;\RR^m)}\leq 2^p C_S^p\eps^{p}\left({1\over \alpha c}(\widecheck{{\rm m}}^\eps_u(A)+\eps)+\int_A|\nabla_\mu u|^pd\mu\right),
$$
which gives
\begin{equation}\label{Diag-SteP3-2}
\lim_{\eps\to0}\limsup_{t\to 1^-}\|u_t^\eps-u\|^p_{L^{\chi p}_\mu(X;\RR^m)}=0
\end{equation}
because $\lim_{\eps\to0}\widecheck{{\rm m}}_u^\eps(A)=\widecheck{{\rm m}}^*_u(A)<\infty$. According to \eqref{Diag-SteP3-1} and \eqref{Diag-SteP3-2}, by diagonalization there exists a mapping $\eps\mapsto t_\eps$, with $t_\eps\to 1^-$ as $\eps\to 0$, such that:
\begin{eqnarray}
&&\lim_{\eps\to 0}\|w_\eps-u\|^p_{L^{\chi p}_\mu(X;\RR^m)}=0;\label{EnD-EqSteP3-1}\\
&&\liminf_{\eps\to 0}E(w_\eps,A)\leq \widecheck{\rm {m}}^*_u(A)\label{EnD-EqSteP3-2}
\end{eqnarray}
with $w_\eps:=u^{\eps}_{t_\eps}$. Since $\chi p\geq p$, $w_\eps\to u$ in $L^{p}_\mu(X;\RR^m)$ by \eqref{EnD-EqSteP3-1}, and \eqref{EEEEqqqq} follows from \eqref{EnD-EqSteP3-2} by noticing that $\overline{E}(u;A)\leq\liminf_{\eps\to 0}E(w_\eps,A)$.
\end{proof}

\medskip 

\paragraph{\bf Step 4: differentiation with respect to \boldmath$\mu$\unboldmath} This step consists of applying Theorem \ref{Vitali-Envelope-Prop2} (with $\Theta=\widecheck{{\rm m}}_u$ where $u\in \mathfrak{S}$). More precisely, we have

\begin{lemma}\label{Lemma-4-MT}
If  \eqref{growth-on-Gx}, \eqref{hyp-convexity} and the left inequality in \eqref{Gx-growth} hold then 
\begin{equation}\label{Step4-IRepTHeo-Equat1}
\widecheck{{\rm m}}^*_u(A)=\int_A\lim_{\rho\to 0}{\widecheck{{\rm m}}_u(Q_\rho(x))\over \mu(Q_\rho(x))}d\mu(x)
\end{equation}
for all $u\in\mathfrak{S}$ and all $A\in\mathcal{O}(X)$. As a consequence, if futhermore \eqref{p-coercivity}, the right inequality in \eqref{Gx-growth} and \eqref{ru-usc-condition} hold then
\begin{equation}\label{Step4-IRepTHeo-Equat2}
\overline{E}(u,A)=\int_A\lim_{\rho\to 0}\limsup_{t\to 1^-}{{\rm m}_{tu}(Q_\rho(x))\over \mu(Q_\rho(x))}d\mu(x)
\end{equation}
for all $u\in\mathfrak{S}$ and all $A\in\mathcal{O}(X)$.
\end{lemma}
\begin{proof}[\bf Proof of Lemma \ref{Lemma-4-MT}]
Fix $u\in\mathfrak{S}$. The integral representation of $\overline{E}(u,\cdot)$ in \eqref{Step4-IRepTHeo-Equat2} follows from \eqref{Step4-IRepTHeo-Equat1} by using Lemma \ref{Lemma-3-MT} and the definition of $\widecheck{{\rm m}}_u$ in \eqref{DeFINItIon-Of-widecheck-m-u}. So, we only need to establish \eqref{Step4-IRepTHeo-Equat1}. For this, it is sufficient to prove that $\widecheck{{\rm m}}_u$ is subadditive and there exists a finite Radon measure $\nu$ on $X$ which is absolutely continuous with respect to $\mu$ such that 
\begin{equation}\label{Vitali-enVeloPe-ApplicatION-Step4-mAjoration}
\widecheck{{\rm m}}_u(A)\leq \nu(A)
\end{equation}
for all $A\in\mathcal{O}(X)$, and then to apply Theorem \ref{Vitali-Envelope-Prop2}. For each $t\in]0,1[$, from the definition of ${\rm m}_{tu}$ in \eqref{DeFINItIon-Of-widecheck-m-u-bis}, it is easy to see that for every $A,B,C\in\mathcal{O}(X)$ with $B,C\subset A$, $B\cap C=\emptyset$ and $\mu(A\setminus B\cup C)=0$, 
$$
{\rm m}_{tu}(A)\leq {\rm m}_{tu}(B)+{\rm m}_{tu}(C),
$$ 
and so 
$$
\limsup_{t\to 1^-}{\rm m}_{tu}(A)\leq \limsup_{t\to 1^-}{\rm m}_{tu}(B)+\limsup_{t\to 1^-}{\rm m}_{tu}(C)\hbox{, i.e., }\widecheck{{\rm m}}_u(A)\leq \widecheck{{\rm m}}_u(B)+\widecheck{{\rm m}}_u(C), 
$$
which shows the subadditivity of $\widecheck{{\rm m}}_u$. On the other hand, Given any $t\in]0,1[$, by using the right inequality in \eqref{Gx-growth} we have
$$
{\rm m}_{tu}(A)\leq \beta\mu(A)+\beta\int_A G(x,t\nabla_\mu u(x))d\mu(x).
$$
But, from \eqref{hyp-convexity} we see that  $G(x,t\nabla_\mu u(x))\leq \gamma\big(1+G(x,\nabla_\mu u(x))+G(x,0)\big)$ for $\mu$-a.a. $x\in X$, hence
$$
{\rm m}_{tu}(A)\leq \beta\mu(A)+\beta\gamma\mu(A)+ \beta\gamma\int_A\Big(G(x,\nabla_\mu u(x))+G(x,0)\Big)d\mu(x)
$$
 Letting $t\to 1^-$, we conclude that
$$
\widecheck{{\rm m}}_u(A)\leq c\left(\mu(A)+ \int_A\Big(G(x,\nabla_\mu u(x))+G(x,0)\Big)d\mu(x)\right).
$$
with $c:=\beta(1+\gamma)$. Thus \eqref{Vitali-enVeloPe-ApplicatION-Step4-mAjoration} is satisfied with the Radon measure $\nu:=c(1+G(x,\nabla_\mu u(x))+G(x,0))d\mu$ which is necessarily finite since $u\in\mathfrak{S}$ and $G(\cdot,0)\in L^1_\mu(X)$ because $G(\cdot,0)\leq\sup_{|\xi|\leq r}G(\cdot,\xi)$ and  $\sup_{|\xi|\leq r}G(\cdot,\xi)\in L^1_\mu(X)$ by \eqref{growth-on-Gx}.
\end{proof}
\medskip

\paragraph{\bf Step 5: removing by affine functions} According to \eqref{Step4-IRepTHeo-Equat2}, the proof of Theorem \ref{MainTheorem} will be completed (see Remark \ref{remark-STeP6-BiS} and also Step 6) if we prove that for each $u\in \mathfrak{S}$ and $\mu$-a.e. $x\in X$, we have:
\begin{eqnarray}
&&\lim_{\rho\to 0}\limsup_{t\to 1^-}{{\rm m}_{tu}(Q_\rho(x))\over \mu(Q_\rho(x))}\geq\liminf_{t\to 1^-}\limsup_{\rho\to 0}{{\rm m}_{tu_x}(Q_\rho(x))\over\mu(Q_\rho(x))},\nonumber\\
&&\hbox{i.e., }\lim_{\rho\to 0}{\widecheck{{\rm m}}_{u}(Q_\rho(x))\over \mu(Q_\rho(x))}\geq\liminf_{t\to 1^-}\limsup_{\rho\to 0}{{\rm m}_{tu_x}(Q_\rho(x))\over\mu(Q_\rho(x))}; \label{FiNaL-EqUa2}\\
&&\lim_{\rho\to 0}\limsup_{t\to 1^-}{{\rm m}_{tu}(Q_\rho(x))\over \mu(Q_\rho(x))}\leq\liminf_{t\to 1^-}\limsup_{\rho\to 0}{{\rm m}_{tu_x}(Q_\rho(x))\over\mu(Q_\rho(x))},\nonumber\\
&&\hbox{i.e., }\lim_{\rho\to 0}{\widecheck{{\rm m}}_{u}(Q_\rho(x))\over \mu(Q_\rho(x))}\leq\liminf_{t\to 1^-}\limsup_{\rho\to 0}{{\rm m}_{tu_x}(Q_\rho(x))\over\mu(Q_\rho(x))},\label{FiNaL-EqUa1}
\end{eqnarray}
where $u_x\in W^{1,p}_\mu(\Omega;\RR^m)$ is given by Proposition \ref{Fundamental-Proposition-for-CalcVar-in-MMS}(d) (and satisfies \eqref{FinALAssuMpTIOnOne} and \eqref{FinALAssuMpTIOnTwo}).

\begin{remark}\label{remark-STeP6-BiS}
In fact, \eqref{FiNaL-EqUa2} and \eqref{FiNaL-EqUa1} means that
$$
\lim_{\rho\to 0}{\widecheck{{\rm m}}_{u}(Q_\rho(x))\over \mu(Q_\rho(x))}=\liminf_{t\to 1^-}\limsup_{\rho\to 0}{{\rm m}_{tu_x}(Q_\rho(x))\over\mu(Q_\rho(x))}.
$$
On the other hand, it is easily seen that
$$
\limsup_{\rho\to 0}{{\rm m}_{tu_x}(Q_\rho(x))\over\mu(Q_\rho(x))}=\mathcal{Q}_\mu L(x,\nabla_\mu u(x)),
$$
where $\mathcal{Q}_\mu L$ is the $\mu$-quasiconvexification of $L$ defined in \eqref{DefInITioN-Of-Hrho-mu-Lx-xi}. Hence
$$
\lim_{\rho\to 0}{\widecheck{{\rm m}}_{u}(Q_\rho(x))\over \mu(Q_\rho(x))}=\mathcal{Q}_\mu L(x,\nabla_\mu u(x)).
$$
\end{remark}

We only give the proof of \eqref{FiNaL-EqUa2}. As the proof of \eqref{FiNaL-EqUa1}  uses the same method, its detailled verification is left to the reader.  

\begin{proof}[\bf Proof of (\ref{FiNaL-EqUa2})]
Fix $u\in\mathfrak{S}$. Fix any $\eps>0$, any $\tau\in]0,1[$, any $\rho\in]0,\eps[$, any $t\in]0,1[$ and any $s\in]t,1[$. By definition of ${\rm m}_{{s}u}(Q_{\tau\rho}(x))$, where there is no loss of generality in assuming that $\mu(\partial Q_{\tau\rho}(x))=0$, there exists $w:X\to\RR^m$ such that $w-{s}u\in W^{1,p}_{\mu,0}(Q_{\tau\rho}(x);\RR^m)$ and 
\begin{equation}\label{FunDaMenTal-IneQualiTY-Final}
\int_{Q_{\tau\rho}(x)} L(y,\nabla_\mu w(y))d\mu(y)\leq {\rm m}_{{s}u}(Q_{\tau\rho}(x))+\eps\mu(Q_{\tau\rho}(x)).
\end{equation}
By Proposition \ref{Fundamental-Proposition-for-CalcVar-in-MMS}(e) there is a Uryshon function $\varphi\in{\rm Lip}(X)$ for the pair $(X\setminus Q_{\rho}(x),\overline{Q}_{\tau\rho}(x))$ such that 
\begin{equation}\label{PlAtEAu-FunCtIOn-ProPerTy}
\|D_\mu\varphi\|_{L^\infty_\mu(X;\RR^N)}\leq {\theta\over\rho(1-\tau)}
\end{equation} 
for some $\theta>0$ (which does not depend on $\rho$). Define $v\in W^{1,p}_{\mu}(Q_\rho(x);\RR^m)$ by
$$
v:=\varphi {t\over s}u+(1-\varphi){t\over s}u_x.
$$
Then $v-{t\over s}u_x\in W^{1,p}_{\mu,0}(Q_\rho(x);\RR^m)$. Using Theorem \ref{cheeger-theorem}(d) and \eqref{Mu-der-Prod} we have
\begin{eqnarray*}
\nabla_\mu(sv)&=&\left\{
\begin{array}{ll}
\nabla_{\mu}(tu)&\hbox{in }\overline{Q}_{\tau\rho}(x)\\
tD_\mu\varphi\otimes(u-u_x)+s\big(\varphi {t\over s}\nabla_\mu u+(1-\varphi){t\over s}\nabla_\mu u(x)\big)&\hbox{in }Q_\rho(x)\setminus \overline{Q}_{\tau\rho}(x)
\end{array}
\right.\\
&=&\left\{
\begin{array}{ll}
\nabla_{\mu}(tu)&\hbox{in }\overline{Q}_{\tau\rho}(x)\\
(1-t){t\over 1-t}D_\mu\varphi\otimes(u-u_x)+t\big(\varphi \nabla_\mu u+(1-\varphi)\nabla_\mu u(x)\big)&\hbox{in }Q_\rho(x)\setminus \overline{Q}_{\tau\rho}(x).
\end{array}
\right.
\end{eqnarray*}

As ${t\over s}w-tu\in W^{1,p}_{\mu,0}(Q_{\tau\rho}(x);\RR^m)$ we have $sv+({t\over s}w-tu)-tu_x\in W^{1,p}_{\mu,0}(Q_\rho(x);\RR^m)$. Noticing that $\mu(\partial Q_{\tau\rho}(x))=0$ and, because of Proposition \eqref{Fundamental-Proposition-for-CalcVar-in-MMS}(a), $\nabla_\mu ({t\over s}w-tu)(y)=0$ for $\mu$-a.e. $y\in Q_\rho(x)\setminus \overline{Q}_{\tau\rho}(x)$, we see that 
\begin{eqnarray*}
{{\rm m}_{tu_x}(Q_{\rho}(x))\over\mu(Q_{\tau\rho}(x))}&\leq&{1\over \mu(Q_{\tau\rho}(x))}\int_{Q_\rho(x)}L\left(y,\nabla_\mu (sv)+\nabla_\mu \Big({t\over s}w-tu\Big)\right)d\mu\label{FiRsTEquAtIOnofTHelaSTPrOOf}\\
&=&{1\over\mu(Q_{\tau\rho}(x))}\int_{\overline{Q}_{\tau\rho}(x)}L\left(y,\nabla_\mu(tu)+\nabla_\mu \Big({t\over s}w-tu\Big)\right)d\mu\nonumber\\
&&+{1\over\mu(Q_{\tau\rho}(x))}\int_{Q_\rho(x)\setminus \overline{Q}_{\tau\rho}(x)}L(y,\nabla_\mu (sv))d\mu\nonumber\\
&=&{1\over\mu(Q_{\tau\rho}(x))}\int_{Q_{\tau\rho}(x)}L\left(y,{t\over s}\nabla_\mu w\right)d\mu\\
&&+{1\over\mu(Q_{\tau\rho}(x))}\int_{Q_\rho(x)\setminus Q_{\tau\rho}(x)}L(y,\nabla_\mu (sv))d\mu.\nonumber\\
\end{eqnarray*}
It follows that
\begin{eqnarray*}
{{\rm m}_{tu_x}(Q_{\rho}(x))\over\mu(Q_{\tau\rho}(x))}&\leq&{1\over \mu(Q_{\tau\rho}(x))}\int_{Q_\rho(x)}L(y,\nabla_\mu w)d\mu\\
&&+\Delta^a_L\left({t\over s}\right)\left({\mu(Q_\rho(x))\over \mu(Q_{\tau\rho}(x))}\mint_{Q_\rho(x)}ad\mu+{1\over \mu(Q_{\tau\rho}(x))}\int_{Q_\rho(x)}L(y,\nabla_\mu w)d\mu\right)\\
&&+{1\over\mu(Q_{\tau\rho}(x))}\int_{Q_\rho(x)\setminus Q_{\tau\rho}(x)}L(y,\nabla_\mu (sv))d\mu,
\end{eqnarray*}
where $a\in L^1_\mu(X;]0,\infty])$ is given by \eqref{ru-usc-condition} (and $\limsup_{r\to 1^-}\Delta_L^a(r)\leq 0$ because $L$ is ru-usc). Taking \eqref{FunDaMenTal-IneQualiTY-Final}, \eqref{hyp-convexity} and the right inequality in \eqref{Gx-growth}  into account we deduce that
\begin{eqnarray*}
{{\rm m}_{tu_x}(Q_{\rho}(x))\over\mu(Q_{\tau\rho}(x))}&\leq &\left(1+\Delta^a_L\left({t\over s}\right)\right)\left({{\rm m}_{su}(Q_{\tau\rho}(x))\over\mu(Q_{\tau\rho}(x))}+\eps\right)\\
&&+\Delta^a_L\left({t\over s}\right){\mu(Q_\rho(x))\over \mu(Q_{\tau\rho}(x))}\mint_{Q_\rho(x)}ad\mu\nonumber\\
&&+{c\over\mu(Q_{\tau\rho}(x))}\int_{Q_\rho(x)\setminus Q_{\tau\rho}(x)}G\left(y,{t\over 1-t}D_\mu\varphi\otimes(u-u_x)\right)d\mu\\
&&+{c\over\mu(Q_{\tau\rho}(x))}\int_{Q_\rho(x)\setminus Q_{\tau\rho}(x)} \big(G(y,\nabla_\mu u)+G(y,\nabla_\mu u(x))\big)d\mu\\
&&+c\left({\mu(Q_\rho(x))\over\mu(Q_{\tau\rho}(x))}-1\right)
\end{eqnarray*}
with $c:=\beta+\beta\gamma+\beta\gamma^2$, where $\gamma>0$ and $\beta>0$ given by \eqref{hyp-convexity} and \eqref{Gx-growth} respectively. Thus, noticing that $\mu(Q_\rho(x))\geq \mu(Q_{\tau\rho}(x))$ and letting $s\to 1^-$, we obtain
\begin{eqnarray}
{{\rm m}_{tu_x}(Q_{\rho}(x))\over\mu(Q_{\rho}(x))}&\leq &\left(1+\limsup_{s\to 1^-}\Delta^a_L\left({t\over s}\right)\right)\left({\widecheck{{\rm m}}_{u}(Q_{\tau\rho}(x))\over\mu(Q_{\tau\rho}(x))}+\eps\right)\nonumber\\
&&+\limsup_{s\to 1^-}\Delta^a_L\left({t\over s}\right){\mu(Q_\rho(x))\over \mu(Q_{\tau\rho}(x))}\mint_{Q_\rho(x)}ad\mu\nonumber\\
&&+{c\over\mu(Q_{\tau\rho}(x))}\int_{Q_\rho(x)\setminus Q_{\tau\rho}(x)}G\left(y,{t\over 1-t}D_\mu\varphi\otimes(u-u_x)\right)d\mu\nonumber\\
&&+{c\over\mu(Q_{\tau\rho}(x))}\int_{Q_\rho(x)\setminus Q_{\tau\rho}(x)} \big(G(y,\nabla_\mu u)+G(y,\nabla_\mu u(x))\big)d\mu\nonumber\\
&&+c\left({\mu(Q_\rho(x))\over\mu(Q_{\tau\rho}(x))}-1\right).\label{FiRsTEquAtIOnofTHelaSTPrOOf}
\end{eqnarray}
On the other hand, by \eqref{PlAtEAu-FunCtIOn-ProPerTy} we have
\begin{eqnarray*}
\left|{t\over 1-t}D_\mu\varphi(y)\otimes(u(y)-u_x(y))\right|&\leq&\left|{t\over 1-t}\right|\|D_\mu\varphi\|_{L^\infty_\mu(X)}\|u-u_x\|_{L^\infty_\mu(Q_\rho(x);\RR^m)}\\
&\leq&{t\theta\over(1-t)(1-\tau)}{1\over\rho}\|u-u_x\|_{L^\infty_\mu(Q_\rho(x);\RR^m)}
\end{eqnarray*}
for $\mu$-a.a. $y\in Q_\rho(x)\setminus Q_{\tau\rho}(x)$. But, since $p>\DiM$, $\lim_{\rho\to 0}{1\over\rho}\|u-u_x\|_{L^\infty_\mu(Q_\rho(x);\RR^m)}=0$ by \eqref{FinALAssuMpTIOnTwo}, hence there exists $\rho_0>0$ (which depends on $t$ and $\tau$) such that
$$
\left|{t\over 1-t}D_\mu\varphi(y)\otimes(u(y)-u_x(y))\right|\leq r
$$
for $\mu$-a.a. $y\in Q_\rho(x)\setminus Q_{\tau\rho}(x)$ and all $\rho\in]0,\rho_0[$ with $r>0$ given by \eqref{growth-on-Gx}. Hence
\begin{equation}\label{StEP5-InEQUaLITY-1}
\int_{Q_\rho(x)\setminus Q_{\tau\rho}(x)}G\left(y,{t\over 1-t}D_\mu\varphi\otimes(u-u_x)\right)d\mu\leq \int_{Q_\rho(x)\setminus Q_{\tau\rho}(x)}\sup_{|\xi|\leq r}G(y,\xi)d\mu(y)
\end{equation}
for all $\rho\in]0,\rho_0[$. Moreover, it easy to see that:
\begin{eqnarray}
\int_{Q_\rho(x)\setminus Q_{\tau\rho}(x)}\sup_{|\xi|\leq r}G(y,\xi)d\mu(y)&\leq&\mu(Q_{\rho}(x))\mint_{Q_{\rho}(x)}\left|\sup_{|\xi|\leq r}G(y,\xi)-\sup_{|\xi|\leq r}G(x,\xi)\right|d\mu(y)\nonumber\\
&&+\mu\left(Q_\rho(x)\setminus Q_{\tau\rho}(x)\right)\sup_{|\xi|\leq r}G(x,\xi);\label{StEP5-InEQUaLITY-0-add}\\
\int_{Q_\rho(x)\setminus Q_{\tau\rho}(x)} G(y,\nabla_\mu u(y))d\mu(y)&\leq& \mu(Q_{\rho}(x))\mint_{Q_{\rho}(x)}\big|G(y,\nabla_\mu u(y))-G(x,\nabla_\mu u(x))\big|d\mu(y)\nonumber\\
&&+\mu\left(Q_\rho(x)\setminus Q_{\tau\rho}(x)\right)G(x,\nabla_\mu u(x));\label{StEP5-InEQUaLITY-2}\\
\int_{Q_\rho(x)\setminus Q_{\tau\rho}(x)}G(y,\nabla_\mu u(x))d\mu(y)&\leq&\mu(Q_{\rho}(x))\mint_{Q_{\rho}(x)}\big|G(y,\nabla_\mu u(x))-G(x,\nabla_\mu u(x))\big|d\mu(y)\nonumber\\
&&+\mu\left(Q_\rho(x)\setminus Q_{\tau\rho}(x)\right)G(x,\nabla_\mu u(x)).\label{StEP5-InEQUaLITY-3}
\end{eqnarray}
Combining \eqref{StEP5-InEQUaLITY-1} with \eqref{StEP5-InEQUaLITY-0-add}, \eqref{StEP5-InEQUaLITY-2} and \eqref{StEP5-InEQUaLITY-3} we deduce that
\begin{eqnarray}
{{\rm m}_{tu_x}(Q_{\rho}(x))\over\mu(Q_{\rho}(x))}&\leq &\left(1+\limsup_{s\to 1^-}\Delta^a_L\left({t\over s}\right)\right)\left({\widecheck{{\rm m}}_{u}(Q_{\tau\rho}(x))\over\mu(Q_{\tau\rho}(x))}+\eps\right)\nonumber\\
&&+\limsup_{s\to 1^-}\Delta^a_L\left({t\over s}\right){\mu(Q_\rho(x))\over \mu(Q_{\tau\rho}(x))}\mint_{Q_\rho(x)}a(y)d\mu(y)\nonumber\\
&&+c{\mu(Q_{\rho}(x))\over\mu(Q_{\tau\rho}(x))}\mint_{Q_{\rho}(x)}\left|\sup_{|\xi|\leq r}G(y,\xi)-\sup_{|\xi|\leq r}G(x,\xi)\right|d\mu(y)\nonumber\\
&&+c{\mu(Q_{\rho}(x))\over\mu(Q_{\tau\rho}(x))}\mint_{Q_{\rho}(x)}\big|G(y,\nabla_\mu u(y))-G(x,\nabla_\mu u(x))\big|d\mu(y)\nonumber\\
&&+c{\mu(Q_{\rho}(x))\over\mu(Q_{\tau\rho}(x))}\mint_{Q_{\rho}(x)}\big|G(y,\nabla_\mu u(x))-G(x,\nabla_\mu u(x))\big|d\mu(y)\nonumber\\
&&+c\left({\mu(Q_\rho(x))\over\mu(Q_{\tau\rho}(x))}-1\right)\sup_{|\xi|\leq r}G(x,\xi)\nonumber\\
&&+2c\left({\mu(Q_\rho(x))\over\mu(Q_{\tau\rho}(x))}-1\right)G(x,\nabla_\mu u(x))\nonumber\\
&&+c\left({\mu(Q_\rho(x))\over\mu(Q_{\tau\rho}(x))}-1\right).\label{FiRsTEquAtIOnofTHelaSTPrOOf-BiS}
\end{eqnarray}
As $\sup_{|\xi|\leq r}G(\cdot,\xi)\in L^1_\mu(X)$ by \eqref{growth-on-Gx}, and $\mu$ is a doubling measure, we have 
\begin{equation}\label{FiRsTEquAtIOnofTHelaSTPrOOf-BiS-EQ0-add}
\lim_{\eta\to 0}\mint_{Q_{\eta}(x)}\left|\sup_{|\xi|\leq r}G(y,\xi)-\sup_{|\xi|\leq r}G(x,\xi)\right|d\mu(y)=0.
\end{equation}
In the same way, as $u\in\mathfrak{S}$, i.e., $G(\cdot,\nabla_\mu u(\cdot))\in L^1_\mu(X)$ (and $\mu$ is a doubling measure) we can assert that
\begin{equation}\label{FiRsTEquAtIOnofTHelaSTPrOOf-BiS-EQ1}
\lim_{\eta\to0}\mint_{Q_\eta(x)}\big|G(y,\nabla_\mu u(y))-G(x,\nabla_\mu u(x))\big|d\mu(y)=0,
\end{equation}
and by \eqref{growth-on-Gx-2} we have
\begin{equation}\label{FiRsTEquAtIOnofTHelaSTPrOOf-BiS-EQ2}
\lim_{\eta\to0}\mint_{Q_\eta(x)}\big|G(y,\nabla_\mu u(x))-G(x,\nabla_\mu u(x))\big|d\mu(y)=0.
\end{equation}
Moreover, since $a\in L^1_\mu(X;]0,\infty])$, it is clear that
\begin{equation}\label{FiRsTEquAtIOnofTHelaSTPrOOf-BiS-EQ3}
\lim_{\eta\to 0}\mint_{Q_\eta(x)}a(y)d\mu(y)=a(x).
\end{equation}
Letting $\rho\to 0$ in \eqref{FiRsTEquAtIOnofTHelaSTPrOOf-BiS} and using \eqref{FiRsTEquAtIOnofTHelaSTPrOOf-BiS-EQ0-add}, \eqref{FiRsTEquAtIOnofTHelaSTPrOOf-BiS-EQ1}, \eqref{FiRsTEquAtIOnofTHelaSTPrOOf-BiS-EQ2} and  \eqref{FiRsTEquAtIOnofTHelaSTPrOOf-BiS-EQ3} we see that
\begin{eqnarray}
\limsup_{\rho\to0}{{\rm m}_{tu_x}(Q_\rho(x))\over\mu(Q_\rho(x))}&\leq &\left(1+\limsup_{s\to 1^-}\Delta^a_L\left({t\over s}\right)\right)\left(\lim_{\rho\to 0}{\widecheck{{\rm m}}_{u}(Q_{\rho}(x))\over\mu(Q_{\rho}(x))}+\eps\right)\nonumber\\
&&+\limsup_{s\to 1^-}\Delta^a_L\left({t\over s}\right)\limsup_{\rho\to 0}{\mu(Q_\rho(x))\over \mu(Q_{\tau\rho}(x))}a(x)\nonumber\\
&&+c\left({\mu(Q_\rho(x))\over\mu(Q_{\tau\rho}(x))}-1\right)\sup_{|\xi|\leq r}G(x,\xi)\nonumber\\
&&+2c\left(\limsup_{\rho\to 0}{\mu(Q_\rho(x))\over\mu(Q_{\tau\rho}(x))}-1\right)G(x,\nabla_\mu u(x))\nonumber\\
&&+c\left(\limsup_{\rho\to 0}{\mu(Q_\rho(x))\over\mu(Q_{\tau\rho}(x))}-1\right).\label{FiRsTEquAtIOnofTHelaSTPrOOf-BiS-BiS}
\end{eqnarray}
Letting $t\to 1^-$ and $\tau\to 1^-$ in \eqref{FiRsTEquAtIOnofTHelaSTPrOOf-BiS-BiS} and using \eqref{DoublINgAssUMpTiON} we deduce that
\begin{eqnarray}
\liminf_{t\to 1^-}\limsup_{\rho\to0}{{\rm m}_{tu_x}(Q_\rho(x))\over\mu(Q_\rho(x))}&\leq &\left(1+\liminf_{t\to 1^-}\limsup_{s\to 1^-}\Delta^a_L\left({t\over s}\right)\right)\left(\lim_{\rho\to 0}{\widecheck{{\rm m}}_{u}(Q_{\rho}(x))\over\mu(Q_{\rho}(x))}+\eps\right)\nonumber\\
&&+\liminf_{t\to 1^-}\limsup_{s\to 1^-}\Delta^a_L\left({t\over s}\right)a(x).\label{FiRsTEquAtIOnofTHelaSTPrOOf-BiS-BiS-BiS}
\end{eqnarray}
But, by diagonalization there exists a mapping $s\mapsto t_s$ with $t_s\to 1^-$ as $s\to 1^-$ such that:
\begin{trivlist}
 \item $\displaystyle\lim_{s\to 1^-}{t_s\over s}=1$;
 \item $\displaystyle\liminf_{t\to 1^-}\limsup_{s\to 1^-}\Delta^a_L\left({t\over s}\right)\leq\limsup_{s\to 1^-}\Delta^a_L\left({t_s\over s}\right)$.
 \end{trivlist} 
But $\limsup_{r\to 1^-}\Delta_L^a(r)\leq 0$ because $L$ is ru-usc, hence 
$$
\liminf_{t\to 1^-}\limsup_{s\to 1^-}\Delta^a_L\left({t\over s}\right)\leq 0,
$$ 
and so from \eqref{FiRsTEquAtIOnofTHelaSTPrOOf-BiS-BiS-BiS} we conclude that
$$
\liminf_{t\to 1^-}\limsup_{\rho\to0}{{\rm m}_{tu_x}(Q_\rho(x))\over\mu(Q_\rho(x))}\leq \lim_{\rho\to 0}{\widecheck{{\rm m}}_{u}(Q_{\rho}(x))\over\mu(Q_{\rho}(x))}+\eps,
$$
and \eqref{FiNaL-EqUa2} follows by letting $\eps\to0$.
\end{proof}

\medskip

\paragraph{\bf Step 6: end of the proof} From \eqref{Gx-growth} we see that 
$$
\alpha\overline{\mathcal{G}}(u)\leq \overline{E}(u,X)\leq\beta\big(1+\overline{\mathcal{G}}(u)\big)
$$
for all $u\in W^{1,p}_\mu(X;\RR^m)$, where $\overline{\mathcal{G}}$ is defined \eqref{Def-Of-LSC-G-For-ENdofTheProof}. Hence, $\overline{E}(u,X)=\infty$ if $u\in W^{1,p}_\mu(X;\RR^m)\setminus\overline{\mathfrak{S}}$, where $\overline{\mathfrak{S}}$ is the effective domain of $\overline{\mathcal{G}}$, and \eqref{Main-Relaxed-InTeGrAl} follows because it is assumed that $\mathfrak{S}=\overline{\mathfrak{S}}$, where $\mathfrak{S}$ denotes the effective domain of $\mathcal{G}$ defined in \eqref{Def-Of-LSC-G-For-ENdofTheProof-BiS}. \endproof


\end{document}